\documentclass[a4paper,twoside,11pt]{article}

\usepackage{color}
\usepackage{appendix}

\usepackage[latin1]{inputenc}
\usepackage{lmodern}
\usepackage[T1]{fontenc}

\usepackage[a4paper,top=3cm,bottom=3cm,left=3cm,right=3cm,
bindingoffset=5mm]{geometry}

\usepackage{amsmath,amssymb,amsthm}
\usepackage{graphicx}

\usepackage{subfigure}
\usepackage{epstopdf}
\usepackage{comment}
\graphicspath{{./Figures/}}

\usepackage{epstopdf}

\newtheorem{thm}{Theorem}[section]
\newtheorem{prop}[thm]{Proposition}
\newtheorem{lem}[thm]{Lemma}
\newtheorem{cor}[thm]{Corollary}

\newtheorem{remark}{Remark}

\renewcommand{\O}{\Omega}

\newcommand\E{{E}}  

\newcommand\Th{{\mathcal T}_h}

\newcommand\Eh{{\mathcal E}_h}

\newcommand\R{\mathbb{R}}
\newcommand\C{\mathbb{C}}
\newcommand\D{\mathbb{D}}

\renewcommand{\P}{{\mathcal P}}  

\def\decapita#1{}
\def\grecomultibold#1#2{\grecobolddef#1\def\secondobold{#2}%
    \ifx#2\finemultibold\let\next\relax\let\secondobold\relax
    \else\let\next\grecomultibold
    \fi\expandafter\next\secondobold}

\def\grecobolddef#1{%
  \edef\dadef{bf\expandafter\decapita\string#1}%
  \expandafter\def\csname\dadef\endcsname{{\neretto #1}}}

\def\neretto#1{\setbox0=\hbox{\mathsurround=0pt$#1$}%
  \kern.02em\copy0 \kern-\wd0
  \kern-.02em\copy0 \kern-\wd0
  \raise.03em\box0 \kern.02em}
\grecomultibold\alpha\gamma\delta\xi\sigma\tau\eta\theta\varepsilon\rho\pi\varphi\varepsilon\psi\finemultibold
\def\div{\mathop{\rm div}\nolimits}
\def\bdiv{\mathop{\bf div}\nolimits}

\def\curl{\mathop{\rm curl}\nolimits}
\def\bcurl{\mathop{\bf curl}\nolimits}

\def\teps{{\bfvarepsilon}}

\def\Ih{{\mathcal I}_h}
\def\IE{{\mathcal I}_E}

\def\wbox#1;#2;{\vbox{\hrule\hbox{\vrule height#1mm\kern#2mm\vrule
  height#1mm}\hrule}}

\let\phi\varphi

\let\b\b

\newcommand{\bba}   {\mathbf{a}}

\newcommand{\bbc}   {\mathbf{c}}

\newcommand{\bbf}   {\mathbf{f}}

\newcommand{\bbn}   {\mathbf{n}}

\newcommand{\bbp}   {\mathbf{p}}
\newcommand{\bbq}   {\mathbf{q}}
\newcommand{\bbr}   {\mathbf{r}}

\newcommand{\bbu}   {\mathbf{u}}
\newcommand{\bbv}   {\mathbf{v}}
\newcommand{\bbw}   {\mathbf{w}}
\newcommand{\bbx}   {\mathbf{x}}

\newcommand{\bbz}   {\mathbf{z}}

\def\C{\mathbb C}

\def\hpoint#1.#2.#3{{\underline{#1}}_{#2}\cdot
 {\underline{\mathop{\smash{#3}\vphantom{{#1}_{#2}}}}}}

\def\npointp#1.#2{{\underline{#1}}\cdot
 {\underline{\mathop{\smash{#2}\vphantom{{#1}}}}}}

\def\beq{\begin{equation}}
\def\enq{\end{equation}}

\def\bfzero{{\bf 0}}

\author{
E. Artioli\thanks{Department of Civil Engineering and Computer Science,
	University of Rome Tor Vergata,
	Via del Politecnico 1, 00133 Rome, Italy,
	{\tt artioli@ing.uniroma2.it}},
S. {d}e Miranda\thanks{DICAM, University of Bologna, Viale Risorgimento 2, 40136 Bologna, Italy,
	{\tt stefano.demiranda@unibo.it}},
C. Lovadina\thanks{
Dipartimento di Matematica, Universit\`a di Milano, Via Saldini 50, 20133 Milano, 
and IMATI del CNR, Via Ferrata 1, 27100 Pavia, Italy,
{\tt carlo.lovadina@unimi.it}},
L. Patruno\thanks{DICAM, University of Bologna, Viale Risorgimento 2, 40136 Bologna, Italy, {\tt luca.patruno@unibo.it}}
}

\date{}

\title{A Family of Virtual Element Methods for Plane Elasticity Problems Based on the Hellinger-Reissner Principle}

\begin{document}

\maketitle

\begin{abstract}

A family of Virtual Element schemes based on the Hellinger-Reissner variational principle is presented. A convergence and stability analysis is rigorously developed. Numerical tests confirming the theoretical predictions are performed. 
 
\end{abstract}

{
\section{Introduction}

In this paper we extend the study presented in our previous paper \cite{ADLP_HR}. More precisely, we design and study higher-order Virtual Element Methods (VEM) to approximate the solution of linear elasticity problems in 2D. We take the Hellinger-Reissner variational principle (see \cite{Hughes:book} or \cite{Bo-Bre-For}, for instance) as the basis of the discretization procedure. As it is well-known, imposing both the symmetry of the stress tensor and the continuity of the tractions at the inter-element is typically a great source of troubles in the framework of classical Galerkin schemes. For example, when Finite Element Methods are employed, one is essentially led to adopt either cumbersome elements, or to relax the stress symmetry (this latter choice means that the underlying variational principle is changed). The reason for this difficulty stands in the local polynomial approximation that can not easily accomplish for both the symmetry and continuity constraints mentioned above. More details about this issue can be found in \cite{Bo-Bre-For}.
 
As in \cite{ADLP_HR}, we exploit the great flexibility of VEM to present alternative methods, which provide symmetric stresses, continuous tractions and are reasonably cheap with respect to the delivered accuracy. VEMs reach this goal by abandoning the local polynomial approximation concept, a feature originally used to design conforming Galerkin schemes on general polytopal meshes, see \cite{volley}. Recently, this property has been found useful, in certain situations,  for the numerical treatment of internal or regularity constraints, such as incompressibility or inter-element regularity (see \cite{BLV}, \cite{BeiraodaVeiga-Manzini:hkp}).      

We also remark that VEM is experiencing a growing interest towards the applications to Structural Mechanics problems (see \cite{GTP14,BeiraoLovaMora,ABLS_part_I,ABLS_part_II,wriggers,BCP,ANR} and \cite{BeiraodaVeiga-Brezzi-Marini:2013,Brezzi-Marini:2012}, for example). Thus, this paper represents a contribution along that line. 

An outline of the paper is as follows.
In Section \ref{sec:1} we introduce the 2D elasticity problem using the mixed Hellinger-Reissner formulation. Section \ref{s:HR-VEM} details the discrete methods, by describing all the relevant projectors, bilinear and linear forms, together with the VEM approximation spaces. The stability and convergence analysis is  developed in Section \ref{s:theoretical}, while numerical experiments are provided in Section \ref{s:numer}. Concluding considerations are given in Section \ref{s:conclusions}.

Throughout the paper, given two quantities $a$ and $b$, we use the notation $a\lesssim b$ to mean: there exists a constant $C$, independent of the mesh-size, such that $a\leq C\, b$. In addition, given a set $\omega\subseteq \R$ (or $\omega\subseteq \R^2$), we denote with $\P_k(\omega)$ the space of polynomials up to degree $k$ defined on $\omega$. Moreover, we use standard notations for Sobolev spaces, norms and semi-norms (cf. \cite{Lions-Magenes}, for example).

\section{The linear 2D elasticity problem}\label{sec:1}

It is well-known, see for example \cite{Bo-Bre-For, Braess:book}, that the linear elasticity problem reads as follows.

 
\begin{equation}\label{strong}
\left\lbrace{
	\begin{aligned}
	&\mbox{Find } (\bfsigma,\bbu)~\mbox{such that}\\
	&-\bdiv \bfsigma= \bbf\qquad \mbox{in $\Omega$}\\
	& \bfsigma = \C \teps(\bbu)\qquad \mbox{in $\Omega$}\\
	&\bbu_{|\partial\Omega}=\bfzero 
	\end{aligned}
} \right.
\end{equation}
where homogeneous Dirichlet boundary conditions are here chosen only for the sake of simplicity. However, different conditions can be treated in standard ways. 
Introducing the $L^2$ scalar product $(\cdot,\cdot)$,
and $a(\bfsigma,\bftau):=(\D \bfsigma, \bftau)$, a mixed variational formulation of the problem is:
\begin{equation}\label{cont-pbl}
\left\lbrace{
\begin{aligned}
&\mbox{Find } (\bfsigma,\bbu)\in \Sigma\times U~\mbox{such that}\\
&a(\bfsigma,\bftau) + (\bdiv \bftau, \bbu)=0 \quad \forall \bftau\in \Sigma\\
& (\bdiv \bfsigma, \bbv) = -(\bbf,\bbv)\quad \forall \bbv\in U .
\end{aligned}
} \right.
\end{equation}
In this paper we confine to consider polygonal domains $\O\subset \R^2$. Furthermore, we set $\Sigma=H(\bdiv;\Omega)$, $U= L^2(\Omega)^2$, and we suppose that $\bbf \in L^2(\O)^2$.
The elasticity fourth-order symmetric tensor $\D:=\C^{-1}$ is assumed to be uniformly bounded, positive-definite and sufficiently regular. 
%
%
%
%
After having introduced a polygonal mesh $\Th$ of meshsize $h$, the bilinear form $a(\cdot,\cdot)$ in~\eqref{cont-pbl} can be split as
\begin{equation}\label{dec_a} a(\bfsigma,\bftau)=\sum_{{\E\in \Th}}a_E(\bfsigma,\bftau) \quad \textrm{ with } \quad
a_E(\bfsigma,\bftau) : = \int_E \D \bfsigma : \bftau \quad \forall \bfsigma,\bftau \in \Sigma .
\end{equation}

Similarly, it holds
\begin{equation}\label{dec_div}
(\bdiv \bftau, \bbv)=\sum_{{\E\in \Th}}(\bdiv \bftau, \bbv)_E \quad \textrm{with} \quad
(\bdiv \bftau,\bbv)_E : = \int_E \ \bdiv \bftau \cdot \bbv\quad \forall (\bftau,\bbv) \in \Sigma\times U .
\end{equation}
The divergence-free space is defined by:

\begin{equation}\label{eq:kernel}
K = \left\{ \bftau\in \Sigma \, :\, (\bdiv \bftau ,\bbv)=0 \quad \forall \bbv\in U \right\}.
\end{equation}

%
%
%
%
%
%

\section{The Virtual Element Methods}
\label{s:HR-VEM}

In this section we define our Virtual Element discretization of Problem \eqref{cont-pbl}. 
Let $\{\mathcal{T}_h\}_h$ be a sequence of decompositions of $\Omega$ into general polygonal elements $E$ with
\[
h_E := {\rm diameter}(E) , \quad
h := \sup_{E \in \mathcal{T}_h} h_E .
\]
We suppose that for all $h$, each element $E$ in $\mathcal{T}_h$ fulfils the following assumptions:
\begin{itemize}
	\item $\mathbf{(A1)}$ $E$ is star-shaped with respect to a ball of radius $ \ge\, \gamma \, h_E$, 
	\item $\mathbf{(A2)}$ the distance between any two vertexes of $E$ is $\ge c \, h_E$, 
\end{itemize}
where $\gamma$ and $c$ are positive constants. The hypotheses above, and in particular $\mathbf{(A2)}$, may be relaxed (see \cite{BLRXX}, where a scalar elliptic model problem in primal form is considered).

\subsection{The local spaces}\label{ss:E-spaces}

We first fix an integer $k\ge 1$.
Given a polygon $E\in\Th$ with $n_E$ edges, we introduce the space of local infinitesimal rigid body motions:

\begin{equation}\label{eq:rigid}
RM(E)=\left\{ \bbr(\bbx) = \bba + b(\bbx -\bbx_C)^\perp \quad \bba\in\R^2, \ b\in \R  \right\}.
\end{equation}
Above, given $\bbc=(c_1,c_2)^T\in\R^2$, $\bbc^\perp$ is the clock-wise rotated vector  $\bbc^\perp=(c_2,-c_1)^T$, and $\bbx_C$ is the centroid of $E$. We also introduce the space

\begin{equation}\label{eq:rigid_orth}
RM_k^\perp(E)=\left\{ \bbp_k(\bbx)\in \P_k(E)^2 \ : \ \int_E \bbp_k\cdot\bbr = 0 \quad \forall \bbr\in RM(E)  \right\}.
\end{equation}
Hence, the following $L^2$-orthogonal decomposition holds:

\begin{equation}\label{eq:decomp}
\P_k(E)^2 = RM(E)\bigoplus RM_k^\perp(E) .
\end{equation}

Our local approximation space for the  stress field is then defined by

\begin{equation}\label{eq:local_stress}
\begin{aligned}
 \Sigma_h(E)=\Big\{ \bftau_h\in & H(\bdiv;E)\ :\ \exists \bbw^\ast\in H^1(E)^2 \mbox{ such that } \bftau_h=\C\teps(\bbw^\ast);\\  &(\bftau_h\,\bbn)_{|e}\in \P_k(e)^2 \quad \forall e\in \partial E;\quad
\bdiv\bftau_h\in \P_k(E)^2 \Big\}.
\end{aligned}
\end{equation}
   
\begin{remark} Alternatively, the space \eqref{eq:local_stress} can be defined as follows.
	\begin{equation}\label{eq:local_stress-alt}
	\begin{aligned}
	 \Sigma_h(E)=\Big\{ \bftau_h\in & H(\bdiv;E)\ :\ \bftau_h=\bftau_h^T;\quad \curl \bcurl(\D \bftau_h) = 0 ; \\  &(\bftau_h\,\bbn)_{|e}\in \P_k(e)^2 \quad \forall e\in \partial E;\quad
	 \bdiv\bftau_h\in \P_k(E)^2 \Big\}.
	\end{aligned}
	\end{equation}
Here above, the equation $\curl \bcurl (\D \bftau_h) = 0$ is to be intended in the distribution sense.	
\end{remark}   

We remark that, due to the decomposition \eqref{eq:decomp}, we may write $\bdiv\bftau_h = \bbr_\bftau + \bbp_\bftau$ for a unique couple $(\bbr_\bftau,\bbp_\bftau)\in RM(E)\times RM_k^\perp(E)$.
   
We now notice that the $RM(E)$-component $\bbr_\bftau$  of $\bdiv\bftau_h$ is completely determined once 
$(\bftau_h\,\bbn)_{|e}:=\bbp_{k,e}\in \P_k(e)^2$ is given for all $e\in\partial E$. Indeed, let us denote with $\bfvarphi:\partial E\to \R^2$ the function such that $\bfvarphi_{|e}:=\bbp_{k,e}$. Using  the obvious compatibility condition and the orthogonal decomposition \eqref{eq:decomp}, we have:

\begin{equation}\label{eq:compat}
\int_E \bbr_\bftau\cdot \bbr=\int_E \bdiv\bftau_h\cdot \bbr =\int_{\partial E}\bftau_h\bbn\cdot \bbr =
\int_{\partial E}\bfvarphi\cdot \bbr \qquad \forall \bbr\in RM(E) ,
\end{equation}  
which allows to compute $\bbr_\bftau$ using the $\bbp_{k,e}$'s. More precisely, setting (cf \eqref{eq:rigid})

\begin{equation}\label{eq:div1}
\bbr_\bftau = \bfalpha_E + \beta_E (\bbx -\bbx_C)^\perp ,
\end{equation}
from \eqref{eq:compat} we infer

\begin{equation}\label{eq:div2}
\left\lbrace{
	\begin{aligned}
&\bfalpha_E =\frac{1}{|E|}\int_{\partial E}\bfvarphi = \frac{1}{|E|}\sum_{e\in\partial E}\int_e \bbp_{k,e} \\
&\beta_E = \frac{1}{ \int_E | \bbx -\bbx_C |^2 }\int_{\partial E} \bfvarphi\cdot (\bbx -\bbx_C)^\perp  = \frac{1}{ \int_E | \bbx -\bbx_C |^2 }\sum_{e\in\partial E}\int_e \bbp_{k,e}\cdot (\bbx -\bbx_C)^\perp.
\end{aligned}
} \right.
\end{equation}

The equations above suggest to take the following functionals as degrees of freedom in $\Sigma_h(E)$.

\begin{itemize}
	
\item For each edge $e\in\partial E$, given $\bftau_h\in \Sigma_h(E)$:

\begin{equation}\label{eq:edge-dofs}
\bftau_h \longrightarrow \int_e \bftau_h\bbn\cdot\bbp_k \qquad\forall \bbp_k \in \P_k(e)^2 .
\end{equation}	

\item In the polygon $E$, given $\bftau_h\in \Sigma_h(E)$:

\begin{equation}\label{eq:div-dofs}
\bftau_h \longrightarrow \int_E \bdiv\bftau_h\cdot\bfpsi_k \qquad\forall \bfpsi_k \in RM_k^\perp(E) .
\end{equation}
	
\end{itemize}

Indeed, we have:

\begin{lem}\label{lm:local-inter-well}
	If $\bftau_h\in\Sigma_h(E)$, then
	
	\begin{equation}\label{eq:interpol_cond}
	\left\lbrace{
		\begin{aligned}
		&\int_e \bftau_h \bbn\cdot\bbp_k = 0\qquad 
		\forall \bbp_k\in \P_k(e)^2 ,\quad 
		\forall e\in \partial E ;\\
		&\int_{E} \bdiv\bftau_h\cdot \bfpsi_k = 0 \qquad \forall \bfpsi_k\in  RM_k^\perp(E) ,
		\end{aligned}
	} \right.
	\end{equation} 
	imply $\bftau_h=\bfzero$.
	
\end{lem}  

\begin{proof}
	
	The first (boundary) conditions of \eqref{eq:interpol_cond} leads to infer, see \eqref{eq:div1} and \eqref{eq:div2}:
	
	\begin{equation}\label{eq:edge-cond}
	\bftau_h\bbn = \bfzero \quad \mbox{on }
	\partial E ,\qquad
	\bdiv\bftau_h \in RM_k^\perp(E) .
	\end{equation} 
	From \eqref{eq:edge-cond} and the second set of conditions in \eqref{eq:interpol_cond}, we deduce $\bdiv\bftau_h=\bfzero$. Therefore, $\bftau_h\in \Sigma_h(E)$ satisfies, $\bftau_h\bbn = \bfzero$ on $\partial E$, and $\bdiv\bftau_h=\bfzero$, which imply $\bftau_h=\bfzero$ (cf. \eqref{eq:local_stress}).
	
\end{proof}

Alternatively, one may consider the following degrees of freedom (useful for the implementation purposes).

\begin{itemize}
	
	\item For each edge $e\in\partial E$, given $\bftau_h\in \Sigma_h(E)$, the first subset of degrees of freedom is the set of values of $\bftau_h\bbn$ at $k+1$ distinct points in $e$ (for instance, the $k+1$ Gauss-Lobatto nodes). Another possible choice, which has been used in our numerical tests of Section \ref{s:numer} is the following: for each $e$, we introduce a local linear coordinate $s\in[-1,1]$; for both components of $\bftau_h\in \Sigma_h(E)$, the degrees of freedom are the $k+1$ coefficients of their expansion with respect to the basis $\{1,s,s^2,\ldots, s^k  \}$.    
	
	These $2(k+1)$ values account for the degrees of freedom described in \eqref{eq:edge-dofs}.

	\item In the polygon $E$, let us choose a basis $\{\bfvarphi_i \}$ ($i=1,\dots,(k+1)(k+2)-3$) for $RM_k^\perp$, see \eqref{eq:decomp}. Then, for each $\bftau_h\in\Sigma_h(E)$, we may write

	\begin{equation}\label{eq:div-dofs-alt}
		 \bdiv\bftau_h = \bfalpha +\beta(\bbx -\bbx_C)^\perp + \sum_{i=1}^{m_k}\gamma_{i}\bfvarphi_i(\bbx) ,
	\end{equation}
	where $m_k := (k+1)(k+2)-3$.
	
	The values $\{\gamma_1,\dots, \gamma_{m_k}\}$ can be taken as the second subset of degrees of freedom. They account for the degrees of freedom described in \eqref{eq:div-dofs}, while $\{\bfalpha,\beta\}$ are computed according with \eqref{eq:div2}.
	
\end{itemize}
   
The local approximation space for the displacement field is simply defined by, see \eqref{eq:rigid}:

\begin{equation}\label{eq:local_displ}
U_h(E)=\Big\{ \bbv_h\in  L^2(E)^2\ :\ \bbv_h\in \P_k(E)^2 \Big\},
\end{equation} 
and a set of degrees of freedom can be defined in a standard way.

We notice that $\dim( \Sigma_h(E))=2(k+1)\,n_E + m_k$, while $\dim(U_h(E))=(k+1)(k+2)$. In Figure \ref{fig:dofs} the local degrees of freedom for stresses and displacements are schematically depicted for $k=1$: arrows represent traction degrees of freedom (cf. \eqref{eq:edge-dofs}), bullets represent the divergence degrees of freedom (cf. \eqref{eq:div-dofs}), crosses represent the displacement degrees of freedom (cf. \eqref{eq:local_displ}).  

\begin{figure}[h!]
	\centering
	\includegraphics[width=0.75\textwidth]{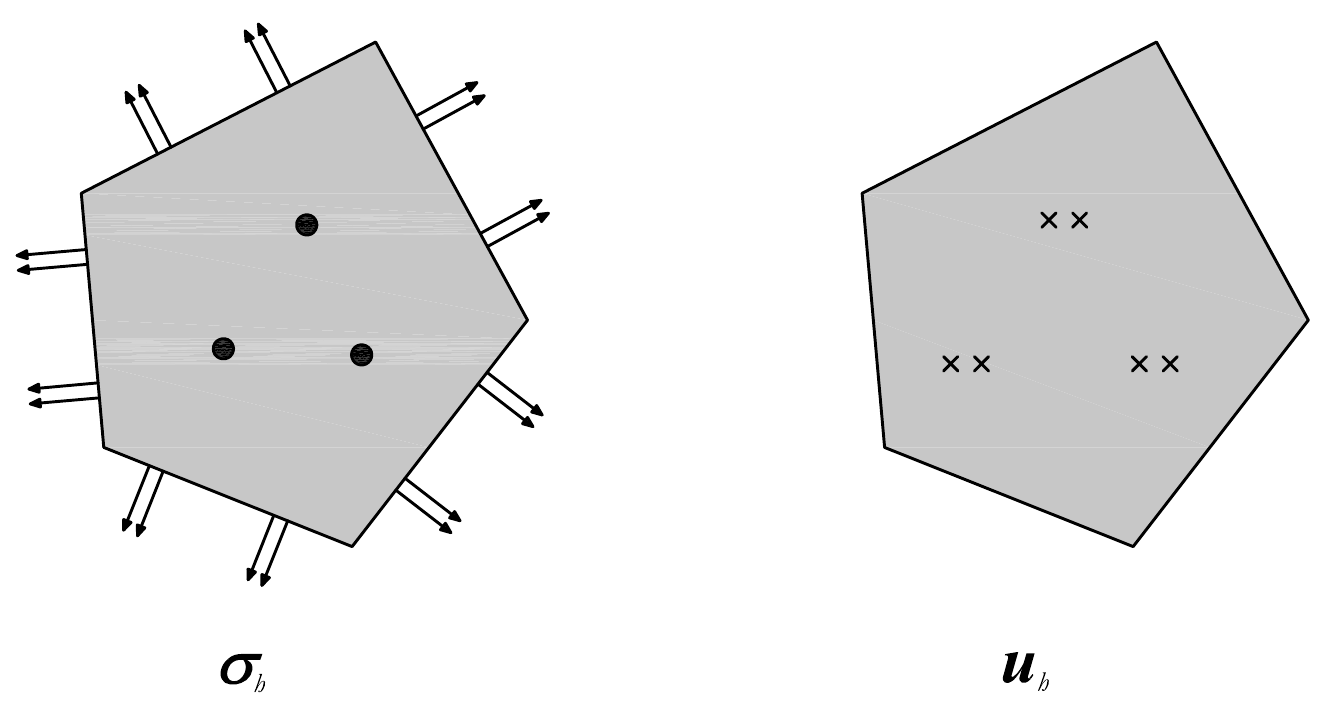}
	\caption{Schematic description of the local degrees of freedom for $k=1$.}
	\label{fig:dofs}
\end{figure}


\subsection{The local bilinear forms}\label{ss:E-bforms}

We begin by noticing that, for every $\bftau_h\in  \Sigma_h(E)$ and $\bbv_h\in U_h(E)$, the term

\begin{equation}\label{eq:div-ex}
\int_E \bdiv \bftau_h\cdot \bbv_h
\end{equation}
is computable. As a consequence, the terms $(\bdiv \bftau, \bbu)$ and $(\bdiv \bfsigma, \bbv)$ in problem \eqref{cont-pbl} are left unaltered. Instead, the term 

\begin{equation}
a_E(\bfsigma_h,\bftau_h)  = \int_E \D \bfsigma_h : \bftau_h
\end{equation}
is not computable for a general couple $(\bfsigma_h,\bftau_h)\in \Sigma_h(E)\times \Sigma_h(E)$. In the spirit of the VEM approach (see \cite{volley}, for instance), we define a suitable approximation $a_E^h(\cdot,\cdot)$. 
Given $E\in \Th$, following \cite{ADLP_HR}, we first introduce the space: 

\begin{equation}\label{eq:sigma-comp}
\widetilde\Sigma(E):=\left\{ \bftau\in H(\bdiv;E)\ : \ \exists \bbw\in H^1(E)^2 \mbox{ \rm such that } \bftau=\C\teps(\bbw) \right\} ,
\end{equation}
and the global space $\widetilde\Sigma$ as 

\begin{equation}\label{eq:sigma-glob}
\widetilde\Sigma:=\left\{ \bftau\in H(\bdiv;\Omega)\ : \ \exists \bbw\in H^1(\O)^2 \mbox{ \rm such that } \bftau=\C\teps(\bbw) \right\} .
\end{equation}

We then introduce the projection operator $\Pi_E^k$ onto the space 

\begin{equation}\label{eq:pspace_0}
T_k(E):=\C\,\teps(\P_{k+1}(E)^2) =\left\{ \C\,\teps( \bbp_{k+1})\ :\ \bbp_{k+1}\in \P_{k+1}(E)^2 \right\} 
\end{equation}

by setting (cf. \eqref{eq:sigma-comp}):
\begin{equation}\label{eq:proj}
\left\{
\begin{aligned}
& \Pi_E^k \, :\,  \widetilde\Sigma(E)\to T_k(E)\\
& \bftau \mapsto \Pi_E^k\bftau\\
&a_E(\Pi_E^k\bftau,\bfpi_k)  = a_E(\bftau,\bfpi_k)\qquad \forall\bfpi_k \in T_k(E) .
\end{aligned}
\right.
\end{equation}
We remark that \eqref{eq:proj} is equivalent to find $\bbp_{k+1}\in \P_{k+1}(E)^2$ such that 

\begin{equation}\label{eq:proj_2}
\int_E \C\,\teps(\bbp_{k+1})\,:\, \teps(\bbq_{k+1}) = \int_E \bftau\,:\, \teps(\bbq_{k+1}) \qquad \forall \, \bbq_{k+1}\in \P_{k+1}(E)^2 .
\end{equation}

\begin{remark}\label{rm:projrem}
Obviously, $\bbp_{k+1}$ is defined up to a term in $RM(E)$. In addition, we observe that for $\C$ constant in $E$, it holds $T_1(E)=\P_1(E)^4_s$; for $\C$ varying in $E$, $T_{k}(E)$ is not even a polynomial space. 
\end{remark}

We have the following proposition.

\begin{prop}\label{pr:projapprox}
	Fix $k\ge 1$, and let $r$ be such that $0\le r\le k+1$.
	Under assumptions $\mathbf{(A1)}$ and $\mathbf{(A2)}$,
	for the projection operator $\Pi_E^k$ defined in \eqref{eq:proj}, the following estimates hold:
	
	\begin{equation}\label{eq:l2est-proj}
		|| \bftau -\Pi_E^k\bftau||_{0,E}\lesssim h_E^r |\bbw|_{r+1,E} \qquad \forall \bftau\in  \widetilde\Sigma(E) \cap H^r(E)^4_s \, \mbox{ with $\bftau=\C\teps(\bbw)$}.
	\end{equation}

\end{prop}

\begin{proof}
If $\bftau\in  \widetilde\Sigma(E)$, there exists (cf. \eqref{eq:sigma-comp}) $\bbw\in H^1(E)^2$ such that $\bftau=\C\teps(\bbw)$. Inspecting \eqref{eq:proj_2}, we realize that $\Pi_E^k\bftau= \C\teps(\bbp_{k+1})$, where $\bbp_{k+1}$ is the Galerkin solution in $\P_{k+1}(E)^2/RM(E)$ of the following Neumann problem.

\begin{equation}\label{eq:proj-neumann}
\left\{
\begin{aligned}
& \mbox{Find $\bbz\in H^1(E)^2/RM(E)$ s.t.:}& \\
& \div \left(\C\teps(\bbz) \right) = \div \left(\C\teps(\bbw)\right) &\mbox{in $E$} \\
& \C\teps(\bbz)\bbn=\C\teps(\bbw)\bbn &\mbox{on $\partial E$}. 
\end{aligned}
\right.
\end{equation}
Estimate \eqref{eq:l2est-proj} now follows from standard arguments of the Galerkin technique combined with polynomial approximation results.   
\end{proof} 

With the operator $\Pi_E^k$ at hand, we set 

\begin{equation}\label{eq:ah1}
\begin{aligned}
a_E^h(\bfsigma_h,\bftau_h)  &= 
a_E(\Pi_E^k\bfsigma_h,\Pi_E^k\bftau_h) + s_E\left( (Id-\Pi_E^k)\bfsigma_h, (Id-\Pi_E^k)\bftau_h \right)\\
&=\int_E \D (\Pi_E^k\bfsigma_h) : (\Pi_E^k\bftau_h)  + s_E\left( (Id-\Pi_E^k)\bfsigma_h, (Id-\Pi_E^k)\bftau_h \right) ,
\end{aligned}
\end{equation}
where $s_E(\cdot,\cdot)$ is a suitable stabilization term. We propose the following choice:
	
	\begin{equation}\label{eq:stab1}
	s_E(\bfsigma_h,\bftau_h) : = \kappa_E\, h_E\int_{\partial E} \bfsigma_h\bbn\cdot \bftau_h\bbn  ,
	\end{equation}
	where $\kappa_E$ is a positive constant to be chosen (for instance, any norm of $\D_{|E}$).

%

%
%
	

		%


\subsection{The local loading terms}\label{ss:E-load}

The loading term, see \eqref{cont-pbl}, is simply:

\begin{equation}\label{eq:fh}
(\bbf,\bbv_h)=\int_{\Omega}\bbf\cdot\bbv_h =\sum_{E\in\Th}\int_{E}\bbf\cdot\bbv_h .
\end{equation}
Computing \eqref{eq:fh} is possible once a suitable quadrature rule is available for polygonal domains (see for instance \cite{NME2759,Sommariva2007,Mousavi2011}).  

%


\subsection{The discrete scheme}\label{ss:discrete}
 
 We introduce a global approximation space for the stress field, by glueing the local approximation spaces, see \eqref{eq:local_stress}:

\begin{equation}\label{eq:global-stress}
\Sigma_h=\Big\{ \bftau_h\in  H(\bdiv;\Omega)\ :\ \bftau_{h|E}\in  \Sigma_h(E)\quad \forall E\in\Th \Big\}.
\end{equation}

For the global approximation of the displacement field, we take, see \eqref{eq:local_displ}:

\begin{equation}\label{eq:global-displ}
U_h=\Big\{ \bbv_h\in  L^2(\Omega)^2\ :\ \bbv_{h|E}\in U_h(E)\quad \forall E\in\Th \Big\}.
\end{equation}
 
In addition, given a local approximation of $a_E(\cdot,\cdot)$, see \eqref{eq:ah1}, we set

\begin{equation}\label{eq:global-ah}
a_h(\bfsigma_h,\bftau_h):= \sum_{E\in\Th}a_E^h(\bfsigma_h,\bftau_h) .
\end{equation}

%
%
%
%
The method we consider is then defined by

\begin{equation}\label{eq:discr-pbl-ls}
\left\lbrace{
	\begin{aligned}
	&\mbox{Find } (\bfsigma_h,\bbu_h)\in \Sigma_h\times U_h~\mbox{such that}\\
	&a_h(\bfsigma_h,\bftau_h)    + (\bdiv \bftau_h, \bbu_h)= 0 \quad \forall \bftau_h\in \Sigma_h\\
	& (\bdiv \bfsigma_h, \bbv_h) = -(\bbf,\bbv_h)\quad \forall \bbv_h\in U_h .
	\end{aligned}
} \right.
\end{equation}
%

%
%
%
%
%
%
%
%
%



\section{Stability and convergence analysis}\label{s:theoretical}

{\em 
Since some results of the analysis follow the guidelines of the theory developed in \cite{ADLP_HR}, we do not provide full details of all the proofs. However, the treatment of the variable material coefficient case is different, and it is reflected in Theorem \ref{th:main_convergence}. Its proof is thus thoroughly provided.}

In the sequel, given a measurable subset $A\subseteq \Omega$ and $r > 2$, we will use the space

\begin{equation}\label{eq:regspace}
	W^r(A):=\left\{  \bftau  \ : \bftau\in L^r(A)^4_s \ , \ \bdiv\bftau\in L^2(A)^2    \right\} ,
\end{equation} 
equipped with the obvious norm.

\subsection{An interpolation operator for stresses}\label{ss:interpol-oper}

We now introduce a local interpolation operator $\IE : W^r(E)\to  \Sigma_h(E)$, the higher-order version of the one introduced in \cite{ADLP_HR}. Given $\bftau\in W^r(E)$, $\IE\bftau\in \Sigma_h(E)$ is determined by:

\begin{equation}\label{eq:loc-interp_0}
\left\{
\begin{aligned}
	& \int_{\partial E} (\IE \bftau) \bbn\cdot \bfvarphi_k = \int_{\partial E}  \bftau\bbn\cdot \bfvarphi_k \qquad \forall \bfvarphi_k\in  R_k(\partial E) \\
	& \int_{E} \bdiv(\IE \bftau)\cdot \bfpsi_k = \int_{E} \bdiv \bftau\cdot \bfpsi_k \qquad \forall \bfpsi_k\in  RM_k^\perp(E) .
\end{aligned}	
\right.
\end{equation} 
Above, the space $R_k(\partial E)$ is defined by:

\begin{equation}\label{eq:Rast} 
	R_k(\partial E) = \left\{  
	\bfvarphi_k\in L^2(\partial E)^2 \,: \, 
	\bfvarphi_{k | e} \in \P_k(e)^2, \ \forall e\in\partial E  \right\}.
\end{equation}
If $\bftau$ is not sufficiently regular, the integral in the right-hand side of \eqref{eq:loc-interp_0} must be seen as a duality between $W^{-\frac{1}{r},r}(\partial E)^2$ and $W^{\frac{1}{r},r'}(\partial E)^2$. If $\bftau$ is a regular function, the above conditions are equivalent to require: 

\begin{equation}\label{eq:loc-interp}
	\left\lbrace{
		\begin{aligned}
			&\int_e (\IE \bftau) \bbn\cdot\bbq_k = \int_e  \bftau\bbn \cdot\bbq_k\qquad 
			\forall \bbq_k\in \P_k(e)^2 ,\quad 
			\forall e\in \partial E ;\\
			&\int_{E} \bdiv(\IE \bftau)\cdot \bfpsi_k = \int_{E} \bdiv \bftau\cdot \bfpsi_k \qquad \forall \bfpsi_k\in  RM_k^\perp(E).
		\end{aligned}
	} \right.
\end{equation} 

We remark  that, from Lemma \ref{lm:local-inter-well}, $\IE\bftau\in\Sigma_h(E)$ is well-defined by conditions \eqref{eq:loc-interp_0}.
The global interpolation operator $\Ih : W^r(\O)\to \Sigma_h$ is then defined by simply glueing the local contributions provided by $\IE$. More precisely, we set $(\Ih\tau)_{|E} :=\IE\bftau_{|E}$ for every $E\in\Th$ and $\bftau\in W^r(\O)$.

The following {\em commuting diagram property} is one of the key points in the analysis of the methods.

\begin{prop}\label{pr:comm-prop}
	Given $k\ge 1$, for the operator $\Ih : W^r(\O)\to \Sigma_h$ introduced above, it holds:
	
\begin{equation}\label{eq:comm-diagr}
\bdiv (\Ih\bftau) = P_h^k(\bdiv\bftau)\qquad \forall\, \bftau\in W^r(\O) ,
\end{equation}	 
where $P_h^k$ denotes the $L^2$-projection operator onto the piecewise polynomial functions of degree $\le k$. 

\end{prop}

\begin{proof} It is sufficient to prove property \eqref{eq:comm-diagr} locally, in each element $E\in\Th$. Fix now $\bbq_k\in\P_k(E)^2$ and $\bftau\in W^r(E)$. By the decomposition \eqref{eq:decomp}, we write $\bbq_k = \bbr + \bfpsi_k$, where $\bbr\in RM(E)$ and $\bfpsi_k\in RM_k^\perp(E)$. We have:
	
\begin{equation}\label{eq:comm-proof1}
\begin{aligned}
\int_E \bdiv\bftau \cdot \bbq_k &= \int_E \bdiv\bftau \cdot \bbr + \int_E \bdiv\bftau \cdot \bfpsi_k & &\\
&= 
\int_{\partial E} \bftau\bbn \cdot \bbr + \int_E \bdiv\bftau \cdot \bfpsi_k& &\\
& = \int_{\partial E} (\IE\bftau)\bbn \cdot \bbr + \int_E \bdiv(\IE\bftau) \cdot \bfpsi_k & & \mbox{ (by \eqref{eq:loc-interp})}\\
& = \int_{E} \bdiv (\IE\bftau) \cdot \bbr + \int_E \bdiv(\IE\bftau) \cdot \bfpsi_k & & \mbox{ (integration by parts)}\\
& = \int_E \bdiv(\IE\bftau) \cdot \bbq_k.  & &
\mbox{ (since $\bbp_k = \bbr + \bfpsi_k$)} 
\end{aligned}
\end{equation}
From \eqref{eq:comm-proof1} and the definition of $L^2$-projection operator, we get $\bdiv (\IE\bftau) = P_h^k(\bdiv\bftau)$ on $E$.
	
\end{proof}

\subsection{Approximation estimates}\label{ss:approx}

For the interpolation operator $\Ih$, using similar steps to the ones detailed in \cite{ADLP_HR}, one can prove the error estimate stated here below.

\begin{prop}\label{pr:approxest}
	Fix $k\ge 1$, and let $r$ be such that $1\le r\le k+1$.
	 Under assumptions $\mathbf{(A1)}$ and $\mathbf{(A2)}$,
	for the interpolation operator $\IE$ defined in \eqref{eq:loc-interp}, the following estimates hold:
	
	\begin{equation}\label{eq:l2est}
		|| \bftau -\IE\bftau||_{0,E}\lesssim h_E^r |\bftau|_{r,E} \qquad \forall \bftau\in  \widetilde\Sigma(E) \cap H^r(E)^4_s ,
	\end{equation}
	
	\begin{equation}\label{eq:divest}
		\begin{aligned}
			|| \bdiv(\bftau -\IE\bftau)||_{r,E}  \lesssim h_E^r |\bdiv\bftau|_{r,E}   \ \
			\forall \bftau\in \widetilde\Sigma(E) \cap H^r(E)^4_s 
			\mbox{ \rm s.t.  $\bdiv\bftau\in H^r(E)^2$}.
		\end{aligned}
	\end{equation}	
	
\end{prop}

We remark that, in particular, estimate \eqref{eq:divest} is a direct consequence of the commuting property \eqref{eq:comm-diagr} and standard approximation results, see \cite{scott-dupont}.

\subsection{The {\em ellipticity-on-the-kernel} and the {\em inf-sup} conditions}\label{ss:elker}

%

We first notice that (see \eqref{eq:global-stress}, \eqref{eq:local_stress} and \eqref{eq:global-displ}, \eqref{eq:local_displ}): 

\begin{equation}\label{eq:kern-incl}
	\bdiv(\Sigma_h)\subseteq U_h .
\end{equation}
As a consequence, introducing the discrete kernel $K_h\subseteq \Sigma_h$:

\begin{equation}\label{eq:div_incl}
	K_h =\{ \bftau_h\in\Sigma_h\, :\, (\bdiv \bftau_h,\bbv_h)=0 \quad \forall \bbv_h\in U_h  \},
\end{equation}
we infer that $K_h\subseteq K$, i.e. $\bftau_h\in K_h$ implies $\bdiv \bftau_h=\bfzero$ (cf. \eqref{eq:kernel}). Hence, it holds:

\begin{equation}\label{eq:l2-hdiv}
	|| \bftau_h ||_\Sigma = ||\bftau_h ||_0\qquad \forall \bftau_h\in K_h .
\end{equation}
This is essentially the property that leads to the {\em ellipticity-on-the-kernel} condition:

\begin{prop}\label{pr:elker}
	For the method described in Section \ref{s:HR-VEM}, there exists a constant $\alpha >0$ such that
	
	\begin{equation}\label{eq:elker}
		a_h(\bftau_h,\bftau_h)\ge \alpha\, || \bftau_h||^2_\Sigma\qquad \forall \bftau_h\in K_h .
	\end{equation}
	
\end{prop}

%

\begin{remark}\label{rm:incopmress_h}
	Our methods satisfy $K_h\subset K$, where $K$ is defined by \eqref{eq:kernel}. As discussed in \cite{Bo-Bre-For}, this property leads to schemes which do not suffer from volumetric locking (see \cite{Hughes:book}, for instance) and can be used also for nearly incompressible materials.
\end{remark}



%
%
%
%

To continue, the following discrete {\em inf-sup} condition is a consequence of the {\em commuting diagram property} (see Proposition \ref{pr:comm-prop}), and of the theory developed in \cite{ADLP_HR}. 

\begin{prop}\label{pr:inf-sup} Fix an integer $k\ge 1$. Suppose that  assumptions $\mathbf{(A1)}$ and $\mathbf{(A2)}$ are fulfilled. There exists $\beta>0$ such that
	
	\begin{equation}\label{eq:inf-sup}
		\sup_{\bftau_h\in \Sigma_h}\frac{(\bdiv \bftau_h,\bbv_h)}{|| \bftau_h||_{\Sigma}}\ge \beta ||\bbv_h||_{U}\qquad \forall\, \bbv_h\in U_h .
	\end{equation}
	
\end{prop}

\subsection{Error estimates}\label{ss:errest}



We need the following estimate, that can be found in \cite{ADLP_HR}. However, we prove it here in details, for the sake of completeness.

\begin{lem}\label{lm:trace-est}
 Under assumptions $\mathbf{(A1)}$ and $\mathbf{(A2)}$, for every $\bftau_h\in \Sigma_h(E)$ it holds
 
\begin{equation}\label{eq:trest1}
h_E^{1/2} || (I - \Pi_E^k)\bftau_h\bbn||_{0,\partial\,E}\lesssim  || (I - \Pi_E^k)\bftau_h ||_{0,E} + h_E || \bdiv ((I - \Pi_E)\bftau_h) ||_{0,E} .
\end{equation} 
 
\end{lem}

\begin{proof}
For $\bftau_\in \Sigma_h(E)$, set $\bfxi_h := (I - \Pi_E^k)\bftau_h$. By assumptions $\mathbf{(A1)}$, take $\bbx_S\in E$ as the center of the circle with respect to which E is star-shaped. Using also $\mathbf{(A2)}$, $E$ can be regularly triangulated by joining $\bbx_S$ and the vertices of $E$, thus obtaining a set of triangles $T_e$, one per each side $e\in \partial E$. For every triangle $T_e$, let $b_e$ be the standard edge bubble for $e$ (i.e. $b_e=4\,\lambda_1\lambda_2$, if the $\lambda_i$'s are the barycentric coordinates of the two vertices of $e$). Finally, define $\bfvarphi\in H^1(E)^2$ by setting $\bfvarphi_{|T_e}= b_e\bfxi_h\bbn $. We have:

\begin{equation}\label{eq:trest2}
\begin{aligned}
|| \bfxi_h\bbn ||_{0,\partial\,E}^2  \lesssim \int_{\partial\,E} \bfxi_h\bbn\cdot \bfvarphi & =
\int_E \bdiv\bfxi_h\cdot\bfvarphi + \int_E \bfxi_h : \teps(\bfvarphi)\\
& \lesssim h_E || \bdiv\bfxi_h ||_{0,E} h_E^{-1}|| \bfvarphi ||_{0,E}
+ || \bfxi_h ||_{0,E} || \teps(\bfvarphi)||_{0,E}\\
&\lesssim \left( 
h_E || \bdiv\bfxi_h ||_{0,E} 
+ || \bfxi_h ||_{0,E} 
\right) h_E^{-1}|| \bfvarphi ||_{0,E}\\
&\lesssim \left( 
h_E || \bdiv\bfxi_h ||_{0,E} 
+ || \bfxi_h ||_{0,E} 
\right) h_E^{-1/2}|| \bfxi_h\bbn ||_{0,\partial\,E} .
\end{aligned}
\end{equation}
Hence, we get
\begin{equation}\label{eq:trest3}
h_E^{1/2}|| \bfxi_h\bbn ||_{0,\partial\,E}\lesssim 
h_E || \bdiv\bfxi_h ||_{0,E} 
+ || \bfxi_h ||_{0,E} ,
\end{equation}
which is exactly \eqref{eq:trest1}.
\end{proof}	

Another useful bound is provided in the lemma that follows.
 
\begin{lem}\label{lm:first-inverse}
	Under assumptions $\mathbf{(A1)}$ and $\mathbf{(A2)}$, for every $\bftau_h\in \Sigma_h(E)$ the following inverse estimate holds
	
	\begin{equation}\label{eq:firstinv1}
	 || \bdiv\, (\Pi_E^k\bftau_h) ||_{0,E}\lesssim h_E^{-1} || \bftau_h ||_{0,E} .
	\end{equation} 
	
\end{lem}

\begin{proof}
Let $\Pi_E^k\, \bftau_h= \C \teps(\bbp_{k+1})$ for a suitable $\bbp_{k+1}\in \P_{k+1}(E)^2$, see \eqref{eq:proj}-\eqref{eq:proj_2}. A direct computation shows that 

\begin{equation}\label{eq:firstinv2}
|| \bdiv\, (\C \teps(\bbp_{k+1})) ||_{0,E}\lesssim 
|\C|_{W^{1,\infty}(E)}\,|| \teps(\bbp_{k+1}) ||_{0,E}
+ |\C|_{L^{\infty}(E)}\,| \teps(\bbp_{k+1}) |_{1,E} .
\end{equation} 

Since $\teps(\bbp_{k+1})$ is a polynomial of degree at most $k$ for each component, using the techniques in \cite{BLRXX}, we get | $ \teps(\bbp_{k+1}) |_{1,E}\lesssim h_E^{-1}|| \teps(\bbp_{k+1}) ||_{0,E}$. Therefore, from \eqref{eq:firstinv2}, we obtain

\begin{equation}\label{eq:firstinv3}
\begin{aligned}
|| \bdiv\, (\C \teps(\bbp_{k+1})) ||_{0,E}&\lesssim \left(
|\C|_{W^{1,\infty}(E)}
+ h_E^{-1}|\C|_{L^{\infty}(E)}\right) || \teps(\bbp_{k+1}) ||_{0,E}\\
&\lesssim 
h_E^{-1} || \teps(\bbp_{k+1}) ||_{0,E}\lesssim 
h_E^{-1} || \bftau_h ||_{0,E},
\end{aligned}
\end{equation}
and the proof is complete.
\end{proof}

We are now ready to present our main convergence result.

\begin{thm}\label{th:main_convergence} Let $k$ be an integer with $k\ge 1$, and $r$ such that $1\le r\le k+1$.
	Let $(\bfsigma,\bbu)\in\Sigma\times U$ be the solution of Problem \eqref{cont-pbl}, and let $(\bfsigma_h,\bbu_h)\in\Sigma_h\times U_h$ be the solution of the discrete problem \eqref{eq:discr-pbl-ls}. Suppose that  assumptions $\mathbf{(A1)}$ and $\mathbf{(A2)}$ are fulfilled.
	Assuming $\bfsigma$ and $\bbu$ sufficiently regular, the following estimate holds true:
	
	\begin{equation}\label{eq:main_conv-est}
	|| \bfsigma - \bfsigma_h||_{\Sigma} + || \bbu - \bbu_h||_U \lesssim h^r\, \Big( | \bfsigma |_{r} + |\bdiv\,\bfsigma|_{r} + |\C |_{W^{r,\infty}}  || \bfsigma ||_{0} + |\bbu|_{r} \Big) .
	\end{equation}
	%
\end{thm}

\begin{proof}
We first consider any approximated material tensor $\C^k_h$, for which it holds

\begin{equation}\label{eq:er-ests5}
h_E\, |\C -\C^k_h|_{W^{1,\infty}(E)} +  ||\C -\C^k_h||_{L^{\infty}(E)}  \lesssim h_E^r |\C |_{W^{r,\infty}(E)}\ ,  \quad 1\le r\le k+1, \quad \forall E\in \Th.
\end{equation} 
For instance, on each $E\in\Th$, one may take $\C^k_h$ as the component-wise averaged Taylor expansion of $\C_h$ (see \cite{Brenner_Scott}, for example).

Take now $\bfsigma_I=\Ih\bfsigma$, and notice that $(\bfsigma_h -\bfsigma_I) \in K_h\subseteq K$. Hence, using \eqref{cont-pbl} and \eqref{eq:discr-pbl-ls}, we infer 

$$
a(\bfsigma, \bfsigma_h -\bfsigma_I ) = a_h(\bfsigma_I, \bfsigma_h -\bfsigma_I) = 0 . 
$$
Therefore, it holds

\begin{equation}\label{eq:er-ests1}
\begin{aligned}
||\bfsigma_h -& \bfsigma_I||_\Sigma^2  =
||\bfsigma_h -\bfsigma_I||_0^2  \lesssim a_h(\bfsigma_h -\bfsigma_I, \bfsigma_h -\bfsigma_I )\\
& = a(\bfsigma, \bfsigma_h -\bfsigma_I ) - a_h(\bfsigma_I, \bfsigma_h -\bfsigma_I)\\
& = a(\bfsigma -\Pi_h\bfsigma_I,\bfsigma_h -\bfsigma_I) - \sum_{E\in\Th}s_E ( (I-\Pi_E^k)\bfsigma_I, (I-\Pi_E^k)(\bfsigma_h -\bfsigma_I)  )\\
& = T_1 + T_2,
\end{aligned}
\end{equation}
where we have denoted with $\Pi_h^k$ the global projector made up by the local contributions $\Pi_E^k$, see \eqref{eq:proj}. For the term $T_1$, we simply have (cf. \eqref{eq:l2est-proj}):

\begin{equation}\label{eq:er-ests2.0}
T_1 \lesssim || \bfsigma -\Pi_h\bfsigma_I  ||_{0} \, || \bfsigma_h -\bfsigma_I  ||_{0} \lesssim h^r \, |\bfsigma|_r
\, || \bfsigma_h -\bfsigma_I  ||_{0}\lesssim h^r \, |\bfsigma|_r
\, || \bfsigma_h -\bfsigma_I  ||_{\Sigma}.
\end{equation}

The term $T_2$ is more involved. Let us treat it locally, on each polygon $E$.
Using Lemma \ref{lm:trace-est}, we get

\begin{equation}\label{eq:er-ests2}
\begin{aligned}
- s_E ( (I-\Pi_E^k)&\bfsigma_I,  (I-\Pi_E^k)(\bfsigma_h -\bfsigma_I)  )\\ &\lesssim h_E^{1/2} || (I - \Pi_E^k)\bfsigma_I\bbn||_{0,\partial\,E}
 \, h_E^{1/2} || (I - \Pi_E^k)(\bfsigma_h-\bfsigma_I)\bbn||_{0,\partial\,E}\\
& \lesssim \Big( || (I - \Pi_E^k)\bfsigma_I ||_{0,E} + h_E || \bdiv ((I - \Pi_E^k)\bfsigma_I) ||_{0,E} \Big) \times \\
&\quad \Big( || (I - \Pi_E^k)(\bfsigma_h-\bfsigma_I) ||_{0,E} + h_E || \bdiv ((I - \Pi_E^k)(\bfsigma_h-\bfsigma_I)) ||_{0,E} \Big)
\end{aligned}
\end{equation}

Recalling that $\bdiv (\bfsigma_h-\bfsigma_I)=\bfzero$ and using Lemma \ref{lm:first-inverse}, we infer

\begin{equation}\label{eq:er-ests3}
 h_E || \bdiv ((I - \Pi_E^k)(\bfsigma_h-\bfsigma_I)) ||_{0,E} \lesssim || \bfsigma_h-\bfsigma_I ||_{0,E} .
\end{equation}

The $L^2$ continuity of $\Pi_E^k$, combined with \eqref{eq:er-ests2} and \eqref{eq:er-ests3}, gives

\begin{equation}\label{eq:er-ests4}
\begin{aligned}
- s_E ( (I-\Pi_E^k)&\bfsigma_I,  (I-\Pi_E^k)(\bfsigma_h -\bfsigma_I)  )\\ 
& \lesssim \Big( || (I - \Pi_E^k)\bfsigma_I ||_{0,E} + h_E || \bdiv ((I - \Pi_E^k)\bfsigma_I) ||_{0,E} \Big) 
 || \bfsigma_h-\bfsigma_I ||_{0,E} \\
& \lesssim \Big( || (I - \Pi_E^k)\bfsigma_I ||_{0,E} + h_E || \bdiv ((I - \Pi_E^k)\bfsigma_I) ||_{0,E} \Big) 
|| \bfsigma_h-\bfsigma_I ||_{\Sigma(E)}. 
\end{aligned}
\end{equation} 

We now estimate $|| (I - \Pi_E^k)\bfsigma_I ||_{0,E} + h_E || \bdiv ((I - \Pi_E^k)\bfsigma_I) ||_{0,E} $. We first have

\begin{equation}\label{eq:er-ests4bis}
\begin{aligned}
|| (I - \Pi_E^k)\bfsigma_I ||_{0,E} &\leq || (I - \Pi_E^k)(\bfsigma_I-\bfsigma) ||_{0,E} + || (I - \Pi_E^k) \bfsigma ||_{0,E} \\
&\lesssim  || \bfsigma_I-\bfsigma ||_{0,E} + h_E^r\,| \bfsigma |_{r,E}\lesssim h_E^r\,| \bfsigma |_{r,E},
\end{aligned}
\end{equation}  
where we have also used estimates \eqref{eq:l2est-proj} and \eqref{eq:l2est}. To treat the term 

\begin{equation} \label{eq:hardterm}
h_E || \bdiv ((I - \Pi_E^k)\bfsigma_I) ||_{0,E}
\end{equation} 
we argue as follows.

Let $\bbq\in \P_{k+1}(E)^2 $ be such that $\C\,\teps(\bbq)=\Pi_E^k\bfsigma_I$. 
We write

\begin{equation}\label{eq:er-ests6}
\begin{aligned}
h_E\, ||& \bdiv ((I  - \Pi_E^k)\bfsigma_I) ||_{0,E} =h_E\, || \bdiv\, \bfsigma_I - \bdiv\,(\C\teps(\bbq)) ||_{0,E} \\
& \leq h_E\, || \bdiv\, (\bfsigma_I - \C^k_h\teps(\bbq)) ||_{0,E}  +h_E\, || \bdiv\,((\C^k_h-\C)\teps(\bbq)) ||_{0,E}= D_1 + D_2 .
\end{aligned}
\end{equation} 
Since $\bdiv\, (\bfsigma_I - \C^k_h\teps(\bbq))$ is polynomial in $E$, the techniques of \cite{BLRXX} can be used to get the inverse estimate

\begin{equation}\label{eq:er-ests7}
D_1\lesssim  || \bfsigma_I - \C^k_h\teps(\bbq) ||_{0,E} .
\end{equation}
Since, using \eqref{eq:l2est-proj}, \eqref{eq:l2est} and \eqref{eq:er-ests5}, we get 

\begin{equation}\label{eq:er-ests8}
\begin{aligned}
|| \bfsigma_I & - \C^k_h\teps(\bbq) ||_{0,E}  \leq || \bfsigma_I - \Pi_E^k\bfsigma_I ||_{0,E} + ||(\C - \C^k_h)\teps(\bbq) ||_{0,E} \\
&  \leq || (I - \Pi_E^k)(\bfsigma_I - \bfsigma ) ||_{0,E} + || (I - \Pi_E^k)\bfsigma ||_{0,E} + h_E^r |\C |_{W^{r,\infty}(E)}\, ||\teps(\bbq)||_{0,E} \\
&\lesssim || \bfsigma_I - \bfsigma  ||_{0,E} + h_E^r | \bfsigma |_{r,E} 
+ h_E^r |\C |_{W^{r,\infty}(E)}\, ||\bfsigma_I||_{0,E}
\lesssim h_E^r (1+ |\C |_{W^{r,\infty}(E)}) | \bfsigma |_{r,E} , 
\end{aligned}
\end{equation}
we infer that it holds

\begin{equation}\label{eq:er-ests9}
D_1\lesssim  h_E^r (1+ |\C |_{W^{r,\infty}(E)}) | \bfsigma |_{r,E}\lesssim  h_E^r  | \bfsigma |_{r,E}.
\end{equation}

To treat $D_2$, a direct computation shows that

\begin{equation}\label{eq:er-ests10}
D_2\lesssim  h_E\, \left( |\C -\C^k_h|_{W^{1,\infty}(E)} || \bfsigma_I ||_{0,E} +  ||\C -\C^k_h||_{L^{\infty}(E)}  | \teps(\bbq) |_{1,E} \right) .
\end{equation} 
Using an inverse estimate for polynomials on polygons, see \cite{BLRXX}, and estimate \eqref{eq:er-ests5}, we get

\begin{equation}\label{eq:er-ests11}
D_2\lesssim  h_E\,  |\C -\C^k_h|_{W^{1,\infty}(E)} || \bfsigma_I ||_{0,E} +  ||\C -\C^k_h||_{L^{\infty}(E)}  | \teps(\bbq) |_{0,E} 
\lesssim  h_E^r |\C |_{W^{r,\infty}(E)} \, || \bfsigma_I ||_{0,E} .
\end{equation} 
From estimate \eqref{eq:l2est} and the triangle inequality we obtain

\begin{equation}\label{eq:er-ests12}
D_2\lesssim   h_E^r |\C |_{W^{r,\infty}(E)} \, \left( || \bfsigma ||_{0,E} +h_E^r \, | \bfsigma |_{r,E} \right) .
\end{equation}

Combining \eqref{eq:er-ests6}, \eqref{eq:er-ests9} and \eqref{eq:er-ests12}, we have

\begin{equation}\label{eq:er-ests13.0}
h_E\, || \bdiv ((I  - \Pi_E^k)\bfsigma_I) ||_{0,E} \lesssim
h_E^r\, \Big( | \bfsigma |_{r,E} + |\C |_{W^{r,\infty}(E)}  || \bfsigma ||_{0,E}  \Big) 
\end{equation}
From estimate \eqref{eq:er-ests4}, \eqref{eq:er-ests4bis} and \eqref{eq:er-ests13.0} we deduce

\begin{equation}\label{eq:er-ests13}
 - s_E ( (I-\Pi_E^k) \bfsigma_I,  (I-\Pi_E^k)(\bfsigma_h -\bfsigma_I)  ) \lesssim h_E^r\, \Big( | \bfsigma |_{r,E} + |\C |_{W^{r,\infty}(E)}  || \bfsigma ||_{0,E}  \Big) 
|| \bfsigma_h-\bfsigma_I ||_{\Sigma(E)}.
\end{equation} 
Summing up all the local contributions (cf. also \eqref{eq:er-ests1}), one gets

\begin{equation}\label{eq:er-ests14}
T_2 \lesssim h^r\, \Big( | \bfsigma |_{r} + |\C |_{W^{r,\infty}}  || \bfsigma ||_{0}  \Big) 
|| \bfsigma_h-\bfsigma_I ||_{\Sigma}.
\end{equation}

Using \eqref{eq:er-ests2.0}, \eqref{eq:er-ests14}, from \eqref{eq:er-ests1} we infer

\begin{equation}\label{eq:er-ests15}
|| \bfsigma_h-\bfsigma_I ||_{\Sigma} \lesssim h^r\, \Big( | \bfsigma |_{r} + |\C |_{W^{r,\infty}}  || \bfsigma ||_{0}  \Big) 
.
\end{equation}
The triangle inequality and estimates \eqref{eq:l2est}-\eqref{eq:divest} now give

\begin{equation}\label{eq:er-ests16}
|| \bfsigma-\bfsigma_h ||_{\Sigma} \lesssim h^r\, \Big( | \bfsigma |_{r} + |\bdiv\,\bfsigma|_{r} + \C |_{W^{r,\infty}}  || \bfsigma ||_{0}  \Big) 
.
\end{equation} 

We now estimate $|| \bbu -\bbu_h ||_U$. We set $\bbu_I$ as the $L^2$-projection of $\bbu$ onto the subspace $U_h$. The {\em inf-sup} condition \eqref{eq:inf-sup} implies that there exists $\bftau_h\in \Sigma_h$ such that

\begin{equation}\label{eq:er-ests17}
|| \bbu_h-\bbu_I ||_{U} \lesssim (\bdiv\bftau_h,\bbu_h-\bbu_I)=
(\bdiv\bftau_h,\bbu_h-\bbu)\ ,  \qquad \bftau_h\lesssim 1.
\end{equation}
Using \eqref{eq:discr-pbl-ls} and \eqref{cont-pbl}, we get

\begin{equation}\label{eq:er-ests18}
\begin{aligned}
|| \bbu_h-\bbu_I ||_{U} \lesssim 
a(\bfsigma,\bftau_h) - a_h(\bfsigma_h,\bftau_h).
\end{aligned}
\end{equation} 
To estimate the right-hand side of \eqref{eq:er-ests18}, we proceed locally on each polygon $E$. We thus have to consider

\begin{equation}\label{eq:er-ests19}
a_E(\bfsigma,\bftau_h) - a_E(\Pi_E^k\bfsigma_h,\Pi_E^k\bftau_h) -
  s_E ( (I-\Pi_E^k) \bfsigma_h,  (I-\Pi_E^k)\bftau_h  ).
\end{equation}
It is immediate to see that

\begin{equation}\label{eq:er-ests20}
a_E(\bfsigma,\bftau_h) - a_E(\Pi_E^k\bfsigma_h,\Pi_E^k\bftau_h) \lesssim ||
\bfsigma -\bfsigma_h||_{\Sigma(E)}\, ||\bftau_h||_{\Sigma(E)} .
\end{equation}
Similarly to \eqref{eq:er-ests2}, we have

\begin{equation}\label{eq:er-ests21}
\begin{aligned}
- s_E ( (I-\Pi_E^k)&\bfsigma_h,  (I-\Pi_E^k)\bftau_h  ) \\
& \lesssim \Big( || (I - \Pi_E^k)\bfsigma_h ||_{0,E} + h_E || \bdiv ((I - \Pi_E^k)\bfsigma_h ||_{0,E} \Big) \times \\
&\quad \Big( || (I - \Pi_E^k)\bftau_h ||_{0,E} + h_E || \bdiv ((I - \Pi_E^k)\bftau_h ||_{0,E} \Big) .
\end{aligned}
\end{equation}
Using the same arguments as in \eqref{eq:er-ests10}-\eqref{eq:er-ests11}, we get

 \begin{equation}\label{eq:er-ests22}
 h_E || \bdiv ((I - \Pi_E^k)\bftau_h ||_{0,E} \lesssim  h_E || \bdiv \bftau_h ||_{0,E} + ||  \bftau_h ||_{0,E} \lesssim ||\bftau_h||_{\Sigma(E)} .
 \end{equation}
Due to the $L^2$-continuity of $\Pi_E^k$, we deduce

 \begin{equation}\label{eq:er-ests23}
|| (I - \Pi_E^k)\bftau_h ||_{0,E} + h_E || \bdiv ((I - \Pi_E^k)\bftau_h) ||_{0,E} \lesssim ||\bftau_h||_{\Sigma(E)} .
\end{equation}
The term $|| (I - \Pi_E^k)\bfsigma_h ||_{0,E} + h_E || \bdiv ((I - \Pi_E^k)\bfsigma_h ||_{0,E}$ can be treated by adding and subtracting $\bfsigma_I$, and then by using \eqref{eq:er-ests3}, \eqref{eq:er-ests4bis} \eqref{eq:er-ests13.0}. Hence, we have

\begin{equation}\label{eq:er-ests24}
\begin{aligned}
|| (I - \Pi_E^k)& \bfsigma_h ||_{0,E} + h_E || \bdiv ((I - \Pi_E^k)\bfsigma_h ||_{0,E}\\ 
& \leq  || (I - \Pi_E^k) (\bfsigma_h -\bfsigma_I)||_{0,E} + h_E || \bdiv ((I - \Pi_E^k)(\bfsigma_h - \bfsigma_I) ||_{0,E}\\
& \qquad + || (I - \Pi_E^k) \bfsigma_I ||_{0,E} + h_E || \bdiv ((I - \Pi_E^k)\bfsigma_I ||_{0,E}\\
& \lesssim ||\bfsigma_h -\bfsigma_I ||_{\Sigma(E)} +
h_E^r\, \Big( | \bfsigma |_{r,E} + |\C |_{W^{r,\infty}(E)}  || \bfsigma ||_{0,E}  \Big).
\end{aligned}
\end{equation} 
From \eqref{eq:er-ests21}, \eqref{eq:er-ests23} and \eqref{eq:er-ests23}, summing over all the contributions and using \eqref{eq:er-ests15}, we deduce

\begin{equation}\label{eq:er-ests25}
- \sum_{E \in \Th} s_E ( (I-\Pi_E^k)\bfsigma_h,  (I-\Pi_E^k)\bftau_h  ) \lesssim h^r\, \Big( | \bfsigma |_{r} + |\C |_{W^{r,\infty}}  || \bfsigma ||_{0}  \Big)||\bftau_h||_{\Sigma}.
\end{equation}
From \eqref{eq:er-ests18}, \eqref{eq:er-ests20} \eqref{eq:er-ests25}, recalling that $||\bftau_h||_{\Sigma}\lesssim 1$, and using \eqref{eq:er-ests16}, we get

\begin{equation}\label{eq:er-ests26}
\begin{aligned}
|| \bbu_h-\bbu_I ||_{U} \lesssim h^r\, \Big( | \bfsigma |_{r} + |\bdiv\,\bfsigma|_{r} + \C |_{W^{r,\infty}}  || \bfsigma ||_{0}  \Big) .
\end{aligned}
\end{equation}
Now, the triangle inequality, standard approximation estimates, together with bounds \eqref{eq:er-ests16} and \eqref{eq:er-ests26}, give \eqref{eq:main_conv-est}.
\end{proof}

%

%

Supposing full regularity of the analytical solution, Theorem \ref{th:main_convergence}
obviously implies the convergence of order $k+1$:

\begin{cor}\label{cor:main_convergence} Let $k$ be an integer with $k\ge 1$.
	Let $(\bfsigma,\bbu)\in\Sigma\times U$ be sufficiently regular. Let $(\bfsigma_h,\bbu_h)\in\Sigma_h\times U_h$ be the solution of the discrete problem \eqref{eq:discr-pbl-ls}. Suppose that  assumptions $\mathbf{(A1)}$ and $\mathbf{(A2)}$ are fulfilled. Then

	\begin{equation}\label{eq:coroll-est}
	|| \bfsigma - \bfsigma_h||_{\Sigma} + || \bbu - \bbu_h||_U \lesssim h^{k+1}.
	\end{equation}
	%
\end{cor}

\begin{remark}\label{rm:split-err}
	Another consequence of $K_h\subset K$ (cf. Remark \ref{rm:incopmress_h})
	is that the error estimate on the stress field does not depend on the displacement approximation space, see \eqref{eq:er-ests16}. For details about such a point in an abstract framework, we refer to \cite{Bo-Bre-For}, for instance.
\end{remark}

\section{Numerical results}\label{s:numer}
In this section the proposed methodology is tested by assessing its accuracy on a selection of problems. Numerical results confirm the soundness of the proposed approach and its optimal performance.

\subsection{Accuracy assessment}
\label{ss:analytical}
Three boundary value problems are considered on the unit square domain $\Omega = [0, 1]^2$ assumed to be in plane strain regime.
Firstly, the material is assumed to be homogeneous and isotropic with material parameters assigned in terms of the Lam\'e constants, here set as $\lambda = 1$ and $\mu = 1$ \cite{BeiraoLovaMora,ABLS_part_I}. A required solution is chosen in terms of displacement fields and the corresponding body load $\bbf$ is consequently obtained as synthetically indicated in the following:
\begin{itemize}
	\item[$\bullet$]
	Test $a$
	\begin{eqnarray}
		\label{eq:test_2a_sol}
		\left\{
		\begin{array}{l}
			u_1 = x^3 - 3 x y^2 \\
			u_2 = y^3 - 3 x^2 y \\
			\bbf = \bfzero
		\end{array}
		\right .
	\end{eqnarray}
	\item[$\bullet$]
	Test $b$
	\begin{eqnarray}
		\label{eq:test_2b_sol}
		\left\{
		\begin{array}{l}
			u_1 = u_2 = \sin(\pi x)  \sin(\pi y) \\
			f_1 = f_2 = -\pi^2 \left[ -(3 \mu + \lambda ) \sin(\pi x) \sin ( \pi y) + ( \mu + \lambda ) \cos ( \pi x) \cos ( \pi y ) \right] .
		\end{array}
		\right.
	\end{eqnarray}
\end{itemize}
It can be noticed that Test $a$ is a problem with polynomial solution, non-homogeneous Dirichlet boundary conditions and zero loading; whereas Test $b$ has a trigonometric solution, homogeneous Dirichlet boundary conditions and trigonometric distributed loads.

In addition, a case with non-homogeneous material properties is analysed and denoted as Test $c$ in the following. Such a case shares the same analytical solution adopted for Test $a$ in terms of displacement fields, but material properties are imposed as:
\begin{equation}
	\lambda = \mu = 1 - d_c^2(\mathbf{x}),
\end{equation}
\noindent being $\mathbf{x}$ the position vector and $d_c^2$ the square distance from the centre of the domain (i.e. the point $\mathbf{x} = [0.5, 0.5]$).

Eight meshes characterized by different element topologies and distortion (see Fig. \ref{fig:meshes}) are adopted in order to assess the robustness of the proposed approach. The first four meshes are {\em structured} and composed of triangles, quadrilaterals, hexagons and a mix of convex and concave quadrilaterals, respectively. Such meshes are denoted in the following by the letter "S". The second four meshes, representing {\em unstructured} analogues of the first ones, are composed of triangles, quadrilaterals, random polygons and a mix of convex and concave hexagons. Such meshes are denoted in the following by the letter "U". The mesh size parameter used in the following numerical tests is chosen as the average edge length, denoted with $\bar{h}_e$. It is noticed that, under mesh assumptions $\mathbf{(A1)}$ and $\mathbf{(A2)}$ and for a quasi-uniform family of mesh, $\bar{h}_e$ is equivalent to both $h_E$ and $h$. It should be remarked that, with the exception of Test $c$, the same meshes and tests have been used by the authors for the assessment of the low order counterpart of the present general formulation, see \cite{ADLP_HR}.

The following error norms are used in order to asses the accuracy and the convergence rate:

\begin{itemize}
	\item[$\bullet$] Discrete error norms for the stress field:
	\begin{equation}
		\label{eq:stress_err_norm}
		E_{\bfsigma}   :=\left( \sum_{e \in \Eh} |e|\int_{e} \kappa\ | (\bfsigma - \bfsigma_h)\bbn  |^2\right)^{1/2} ,
	\end{equation}
	where $\kappa=\frac{1}{2} {\rm tr}(\D)$.
	We remark that the quantity above scales like the internal elastic energy, with respect to the size of the domain and of the elastic coefficients.  	
	
	We make also use of the $L^2$ error on the divergence:
	
	\begin{equation}
		\label{eq:stress_div_err_norm}
		E_{\bfsigma, \bdiv}   :=\left( \sum_{E \in \Th} \int_E | \bdiv(\bfsigma - \bfsigma_h)  |^2\right)^{1/2} .
	\end{equation}

	\item[$\bullet$] $L^2$ error norm for the displacement field:
	\begin{equation}
		\label{eq:displ_err_norm}
		E_{\bbu} :=\left( \sum_{E \in \Th} \int_{E}  | \bbu - \bbu_h |^2 \right)^{1/2} = ||\bbu - \bbu_h ||_0.
	\end{equation}
	
\end{itemize}

\begin{figure}[h!]
	\centering
	\renewcommand{\thesubfigure}{}
	\subfigure[Tri (S)]{\includegraphics[width=0.24\textwidth,trim = 20mm 3mm 25mm 3mm, clip]{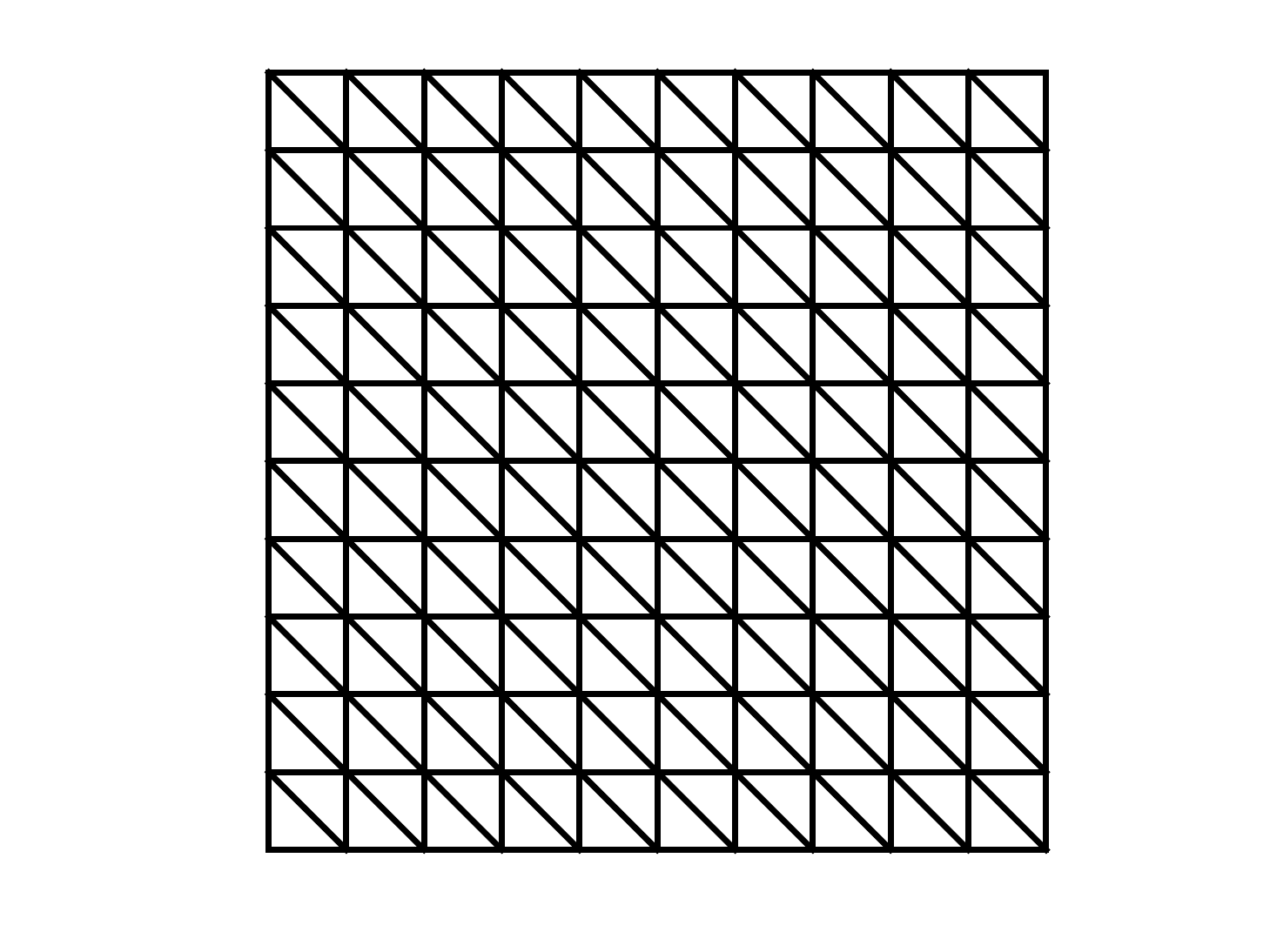}}
	\subfigure[Quad (S)]{\includegraphics[width=0.24\textwidth,trim = 20mm 3mm 25mm 3mm, clip]{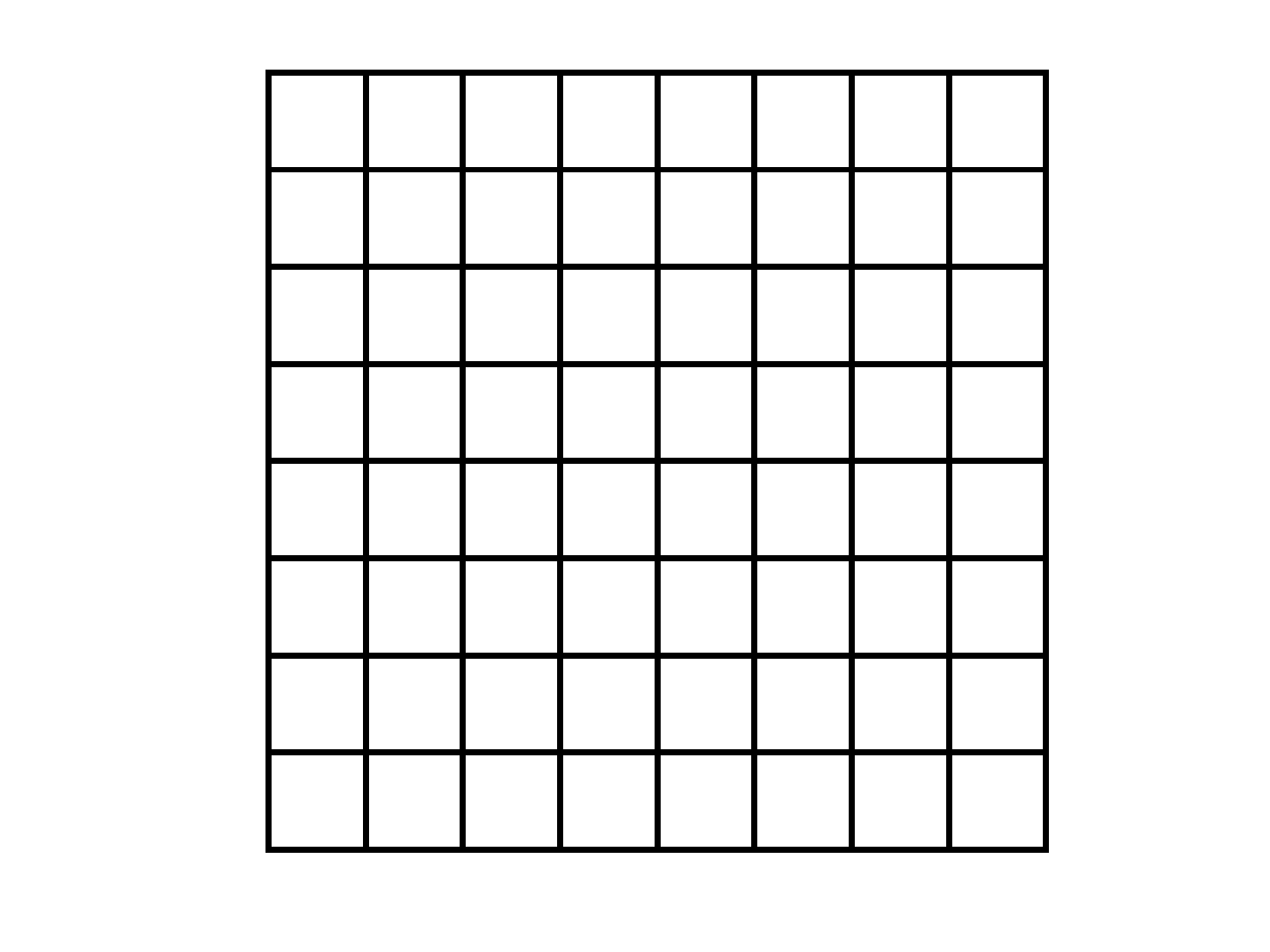}}
	\subfigure[Hex (S)]{\includegraphics[width=0.24\textwidth,trim = 20mm 3mm 25mm 3mm, clip]{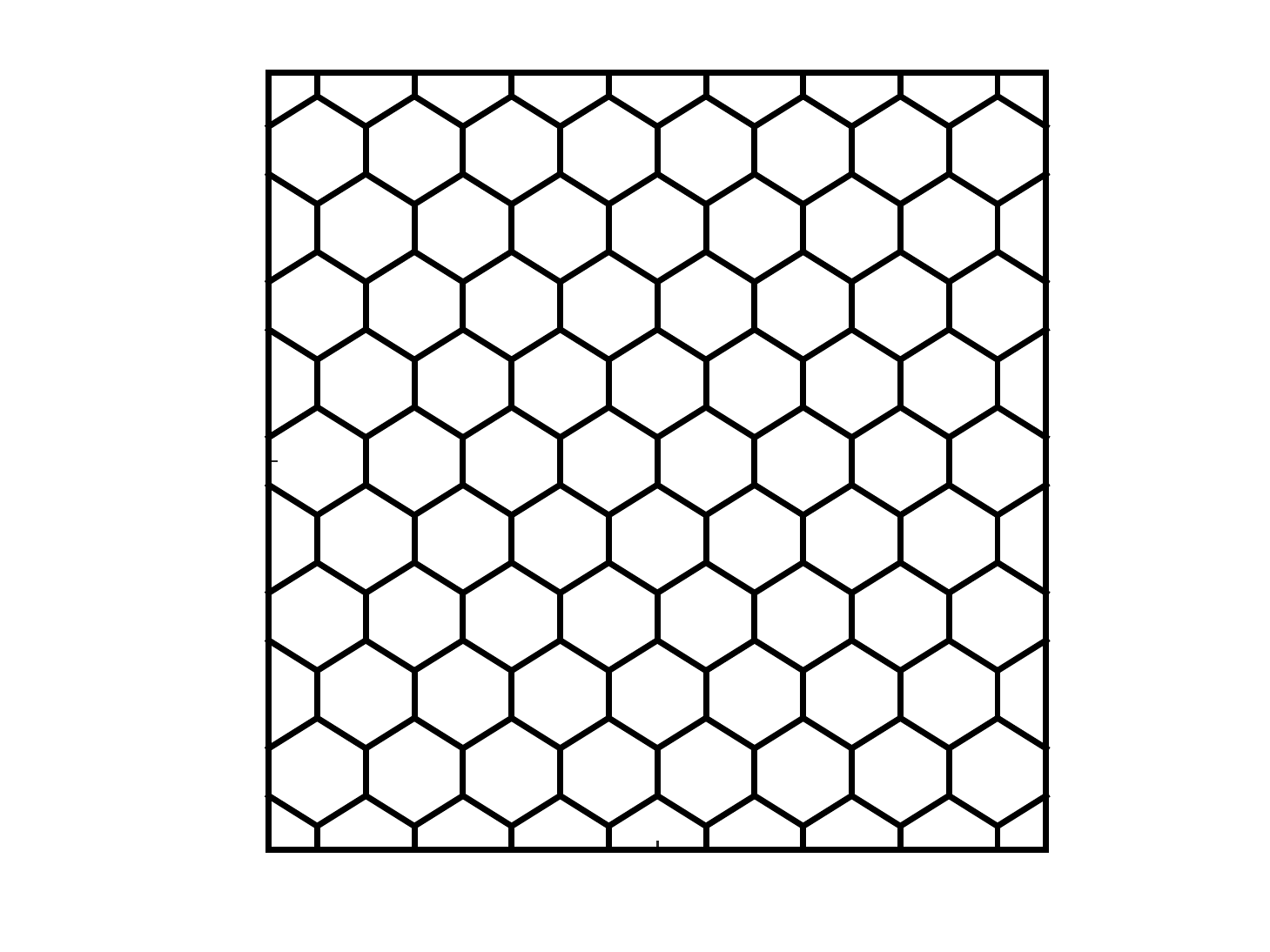}}
	\subfigure[Conc (S)]{\includegraphics[width=0.24\textwidth,trim = 20mm 3mm 25mm 3mm, clip]{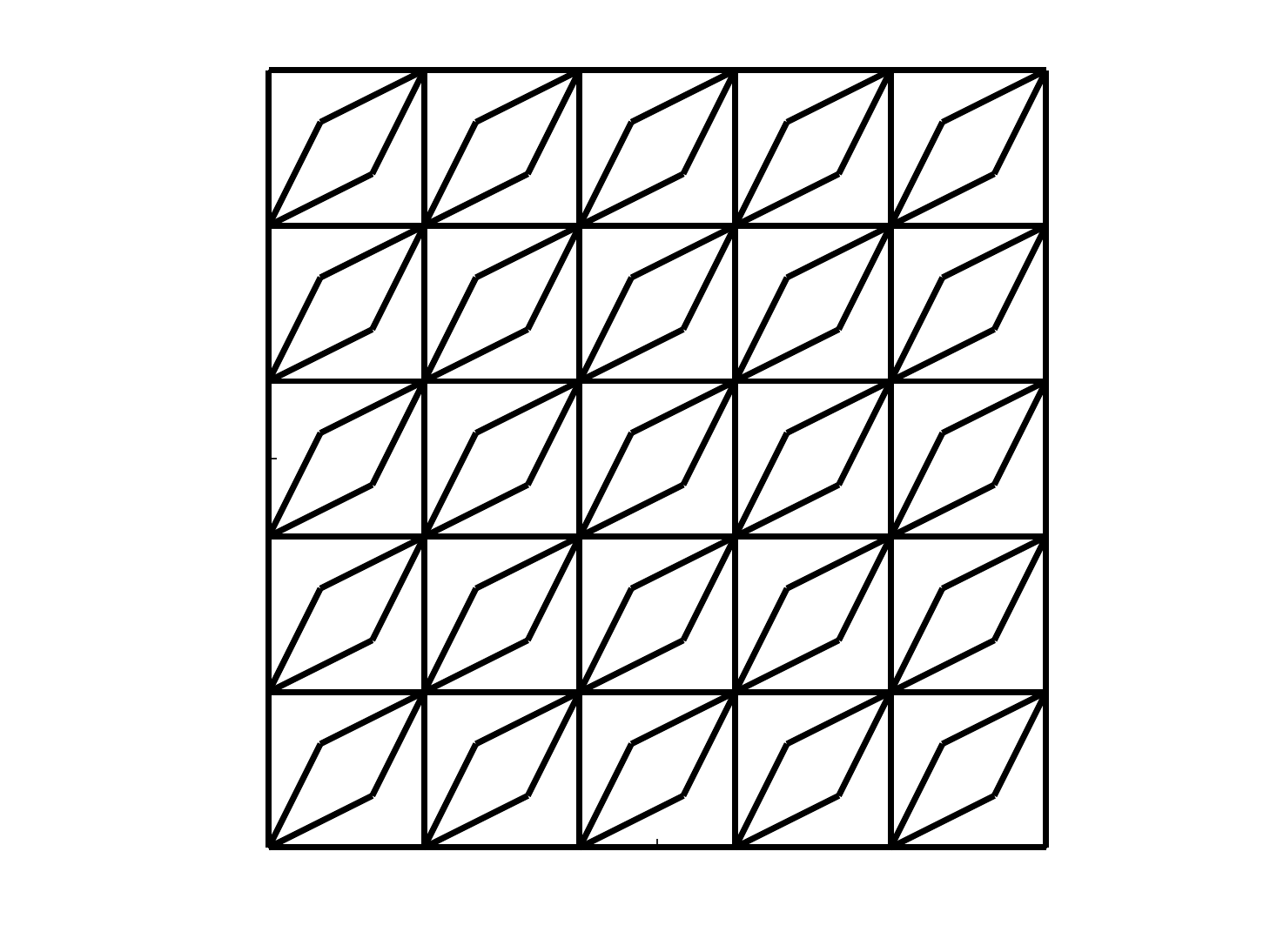}}
	\subfigure[Tri (U)]{\includegraphics[width=0.24\textwidth,trim = 20mm 3mm 25mm 3mm, clip]{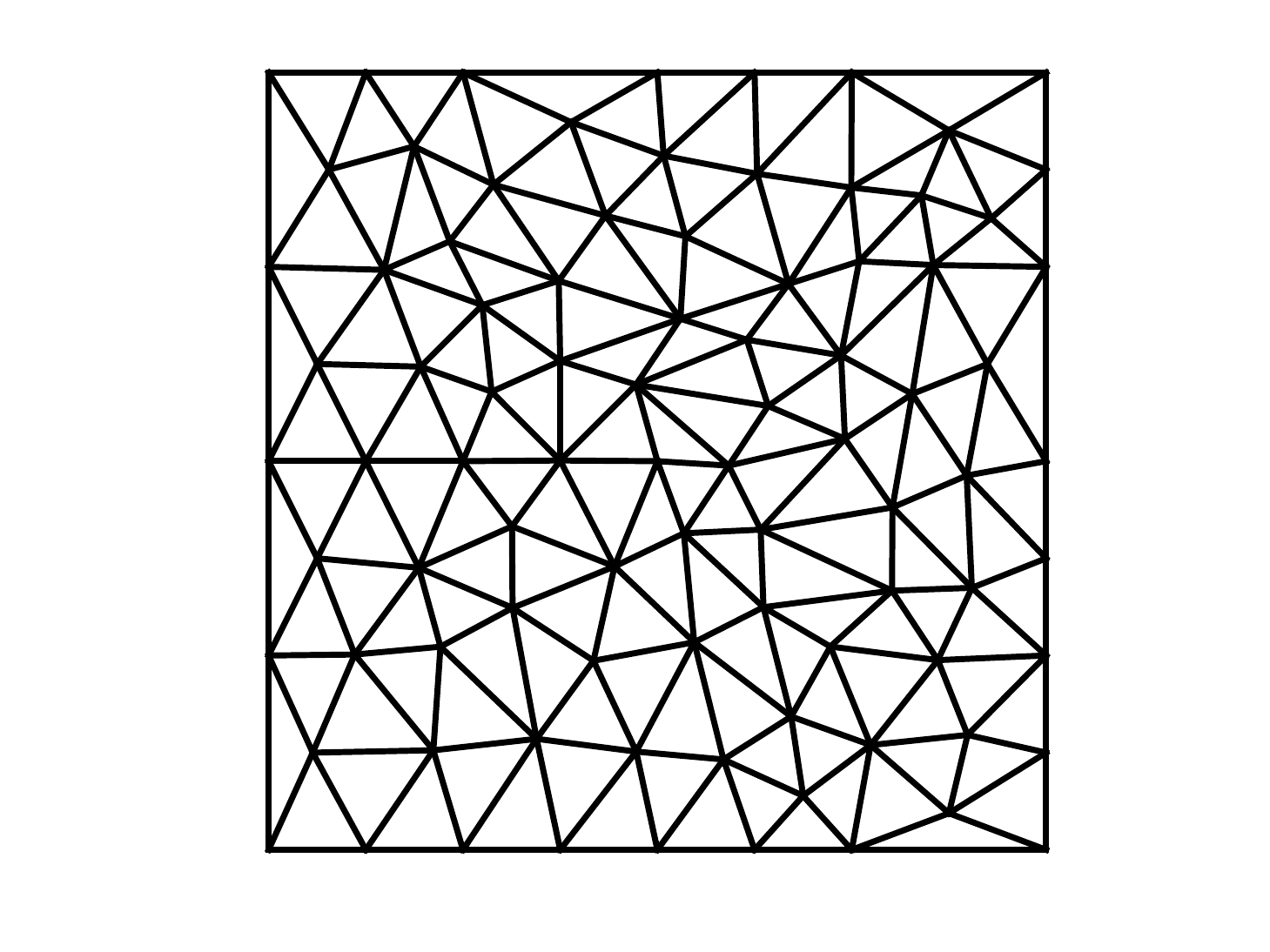}}
	\subfigure[Quad (U)]{\includegraphics[width=0.24\textwidth,trim = 20mm 3mm 25mm 3mm, clip]{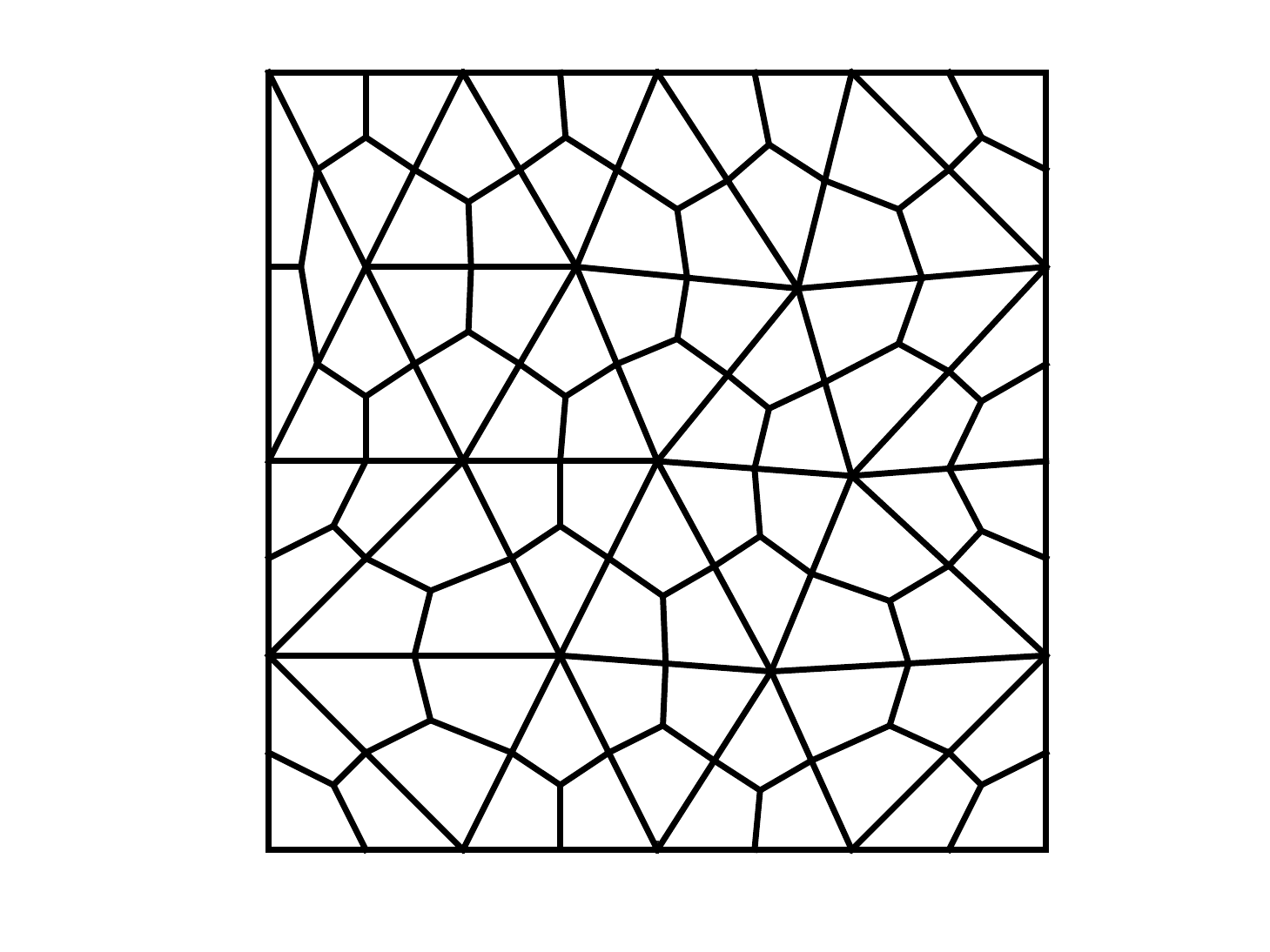}}
	\subfigure[Poly (U)]{\includegraphics[width=0.24\textwidth,trim = 20mm 3mm 25mm 3mm, clip]{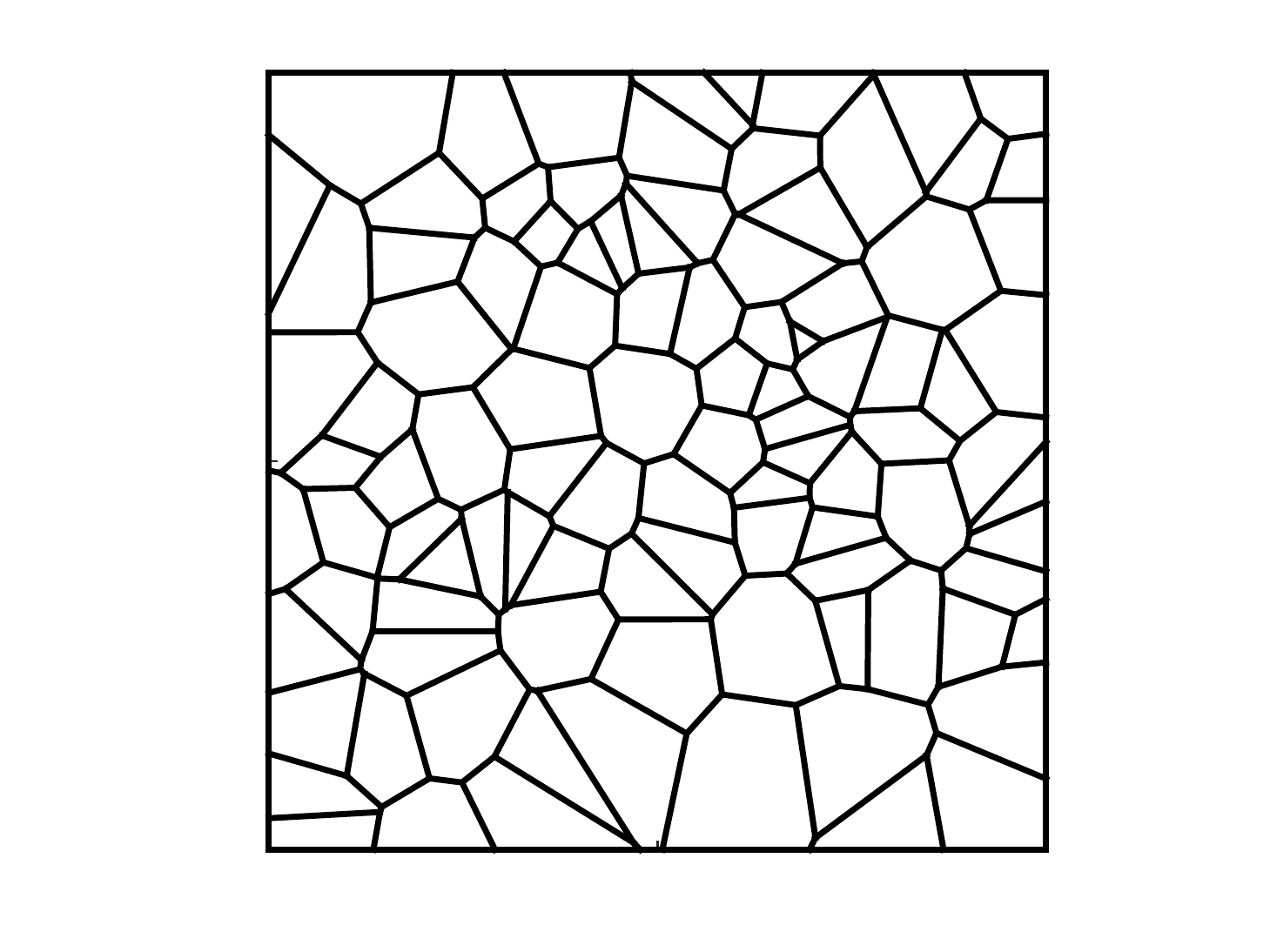}}
	\subfigure[Conc (U)]{\includegraphics[width=0.24\textwidth,trim = 20mm 3mm 25mm 3mm, clip]{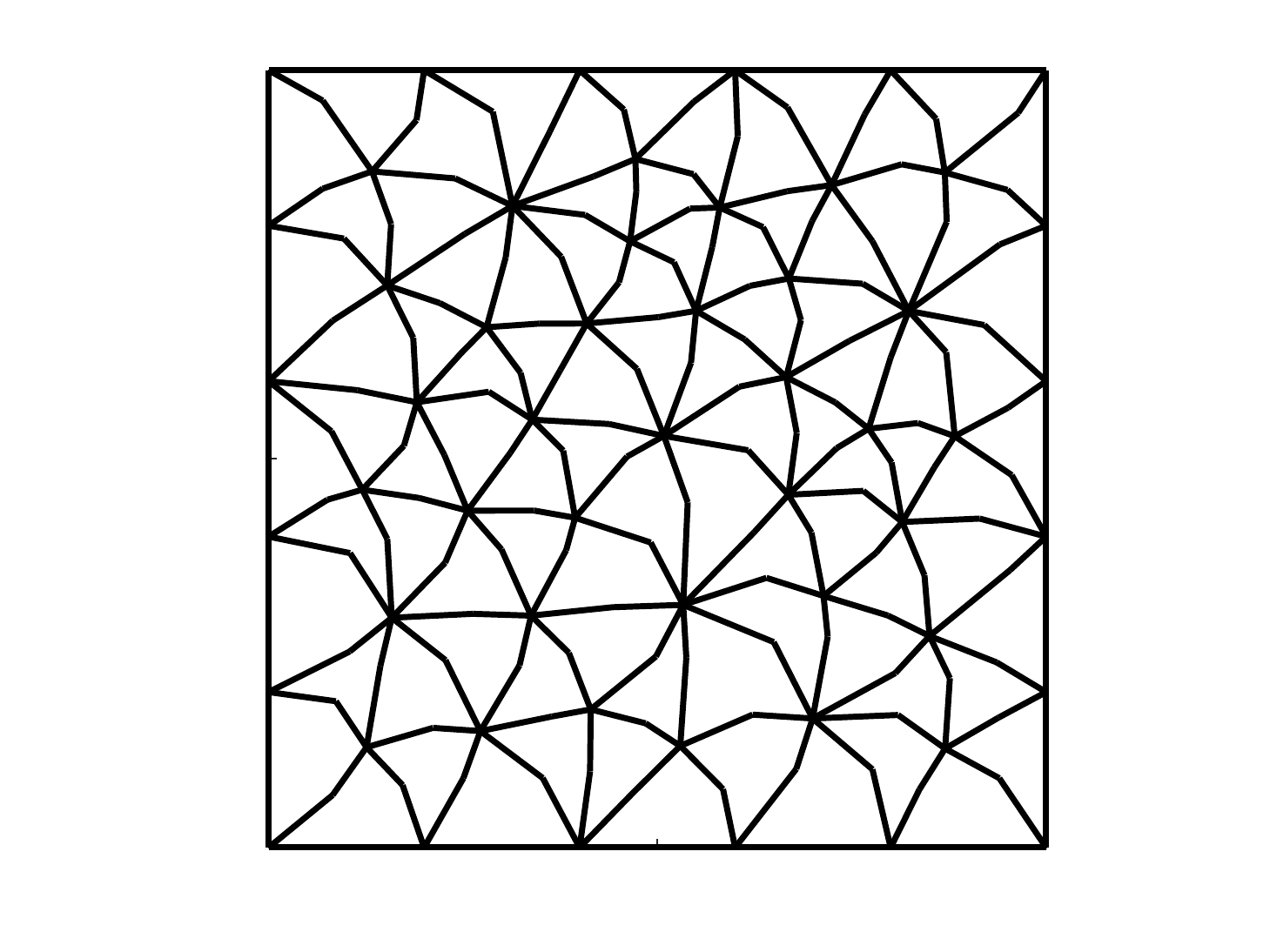}}
	\caption{Overview of the adopted meshes in the numerical tests for the convergence assessment.}
	\label{fig:meshes}
\end{figure}

\subsection{Results for $k=1$}
Figure \ref{fig:resuTestA} reports the $\bar{h}_e-$convergence of the proposed method for Test $a$ when $k$ is equal to 1. The asymptotic convergence rate is approximately equal to $2$ for all the considered error norms and meshes, as expected. The $E_{\bfsigma, \bdiv}$ plots are not reported for this case because such a quantity is captured up to machine precision for all the considered computational grids.

\begin{figure}[h!]
	\centering
	\subfigure[]{\includegraphics[width=0.45\textwidth,trim = 0mm 0mm 0mm 0mm, clip]{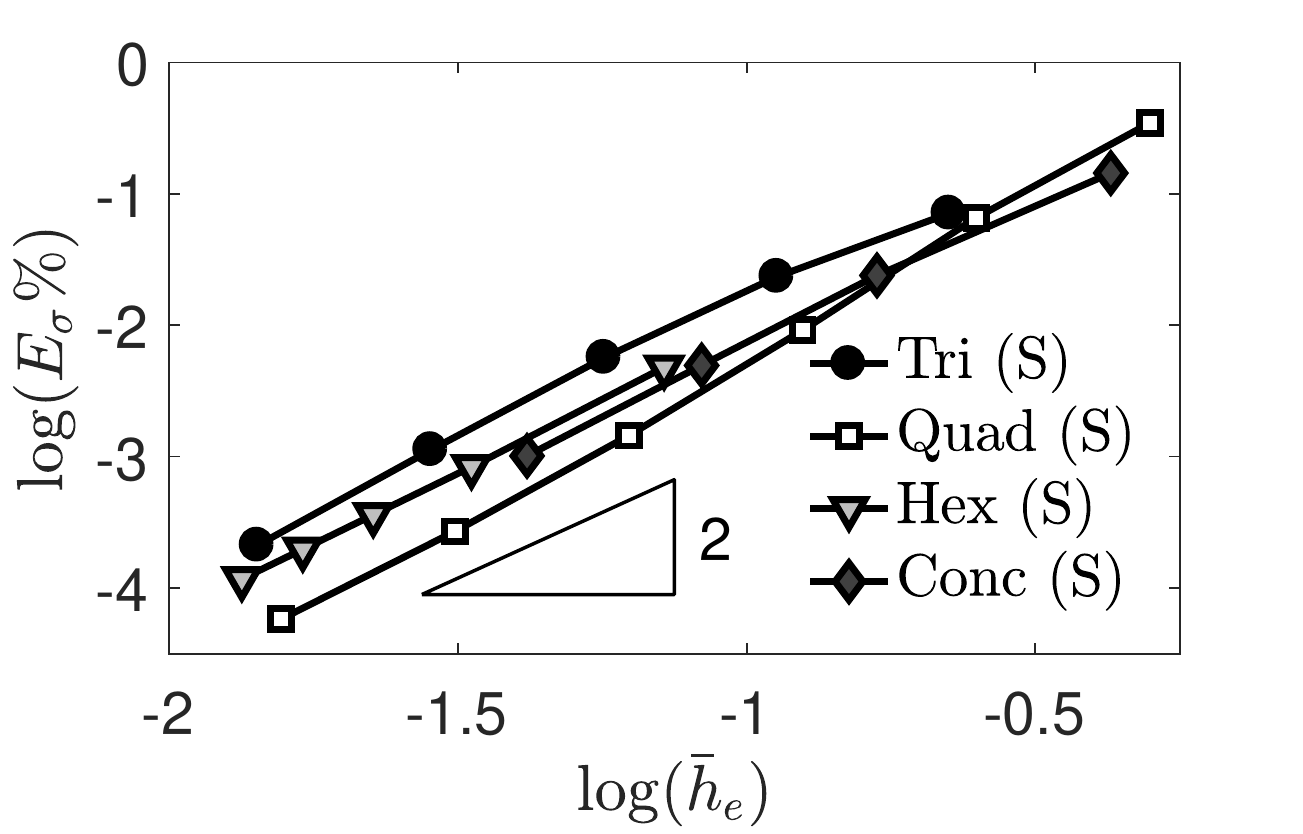}}
	\subfigure[]{\includegraphics[width=0.45\textwidth,trim = 0mm 0mm 0mm 0mm, clip]{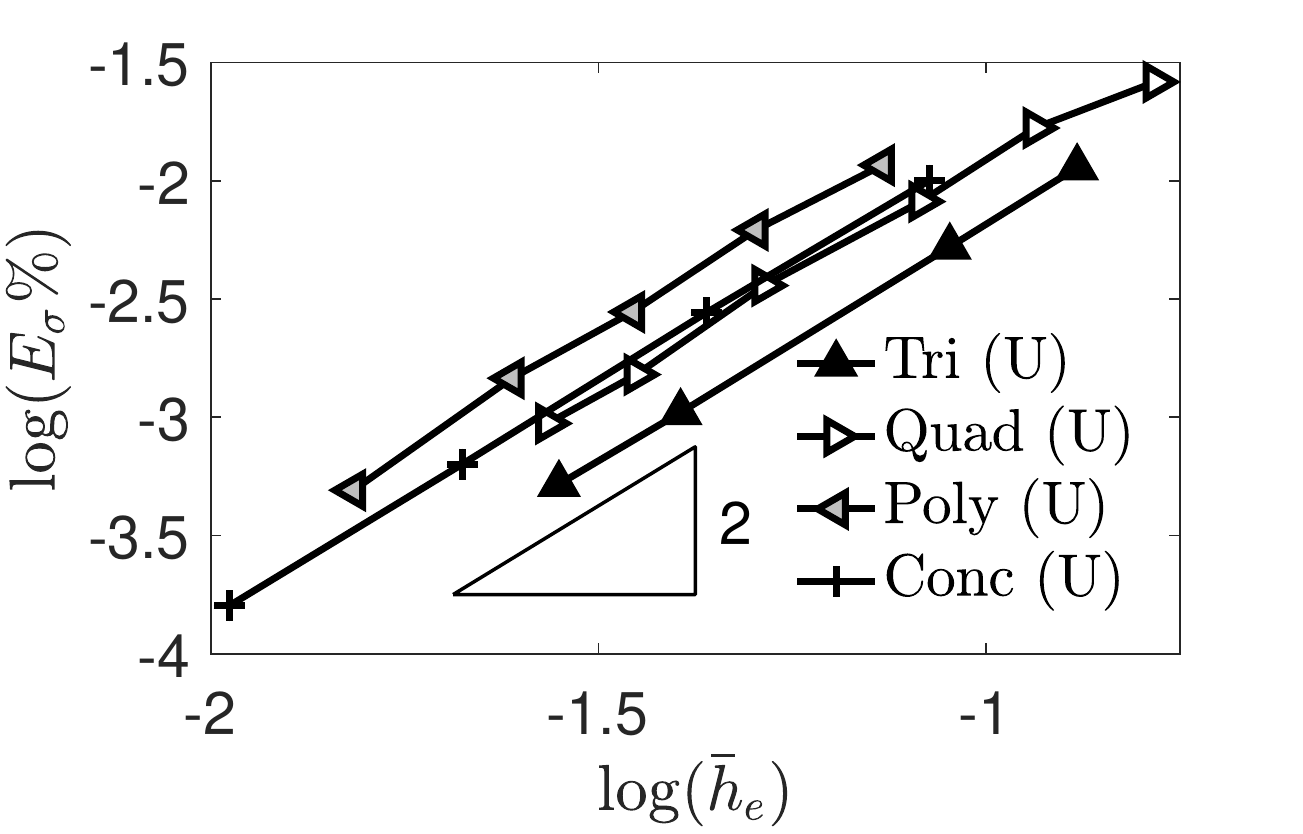}}\\
	\subfigure[]{\includegraphics[width=0.45\textwidth,trim = 0mm 0mm 0mm 0mm, clip]{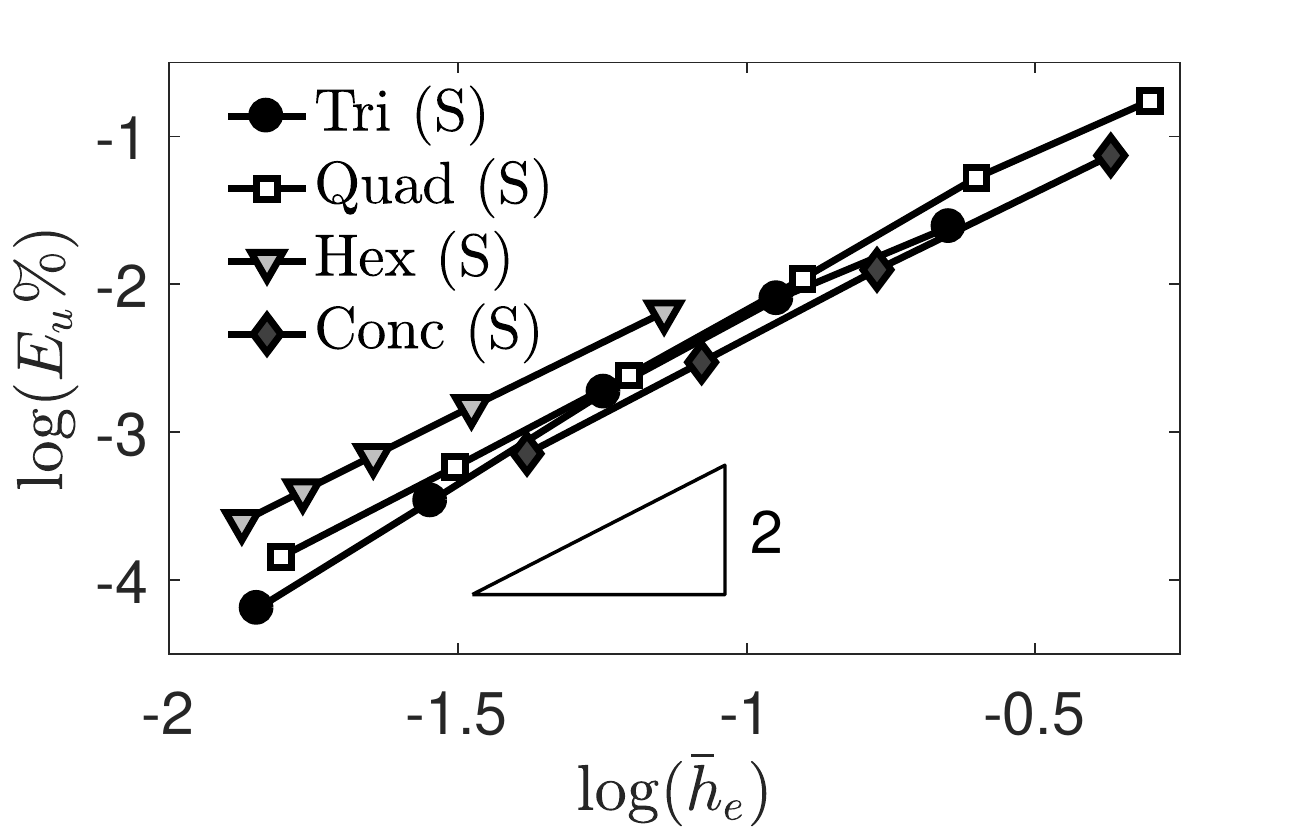}}
	\subfigure[]{\includegraphics[width=0.45\textwidth,trim = 0mm 0mm 0mm 0mm, clip]{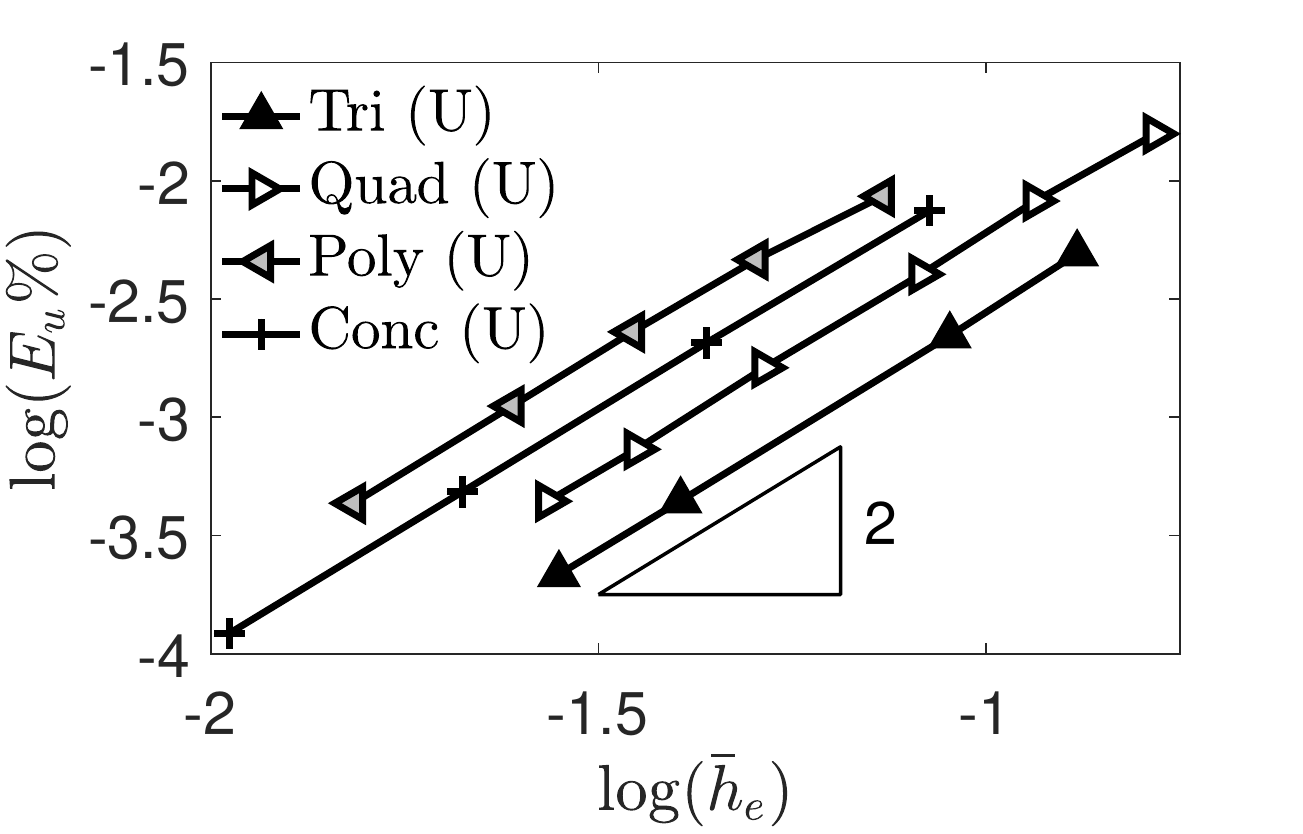}}
	\caption{$\bar{h}_e-$convergence results for Test $a$ on structured and unstructured meshes for $k=1$: (a) and (b) $E_{\bfsigma}$ error norm plots, (c) and (d) $E_{\bbu}$ error norm plots.}
	\label{fig:resuTestA}
\end{figure}

Figures \ref{fig:resuTestB} and \ref{fig:resuTestAvar} report $\bar{h}_e-$convergence for Test $b$ and Test $c$. Asymptotic converge rate is approximately equal to $2$ for all investigated mesh types and error measures. These results highlight the expected optimal performance of the proposed VEM approach and its robustness with respect to the adopted computational grid.

\begin{figure}[h!]
	\centering
	\subfigure[]{\includegraphics[width=0.45\textwidth,trim = 0mm 0mm 0mm 0mm, clip]{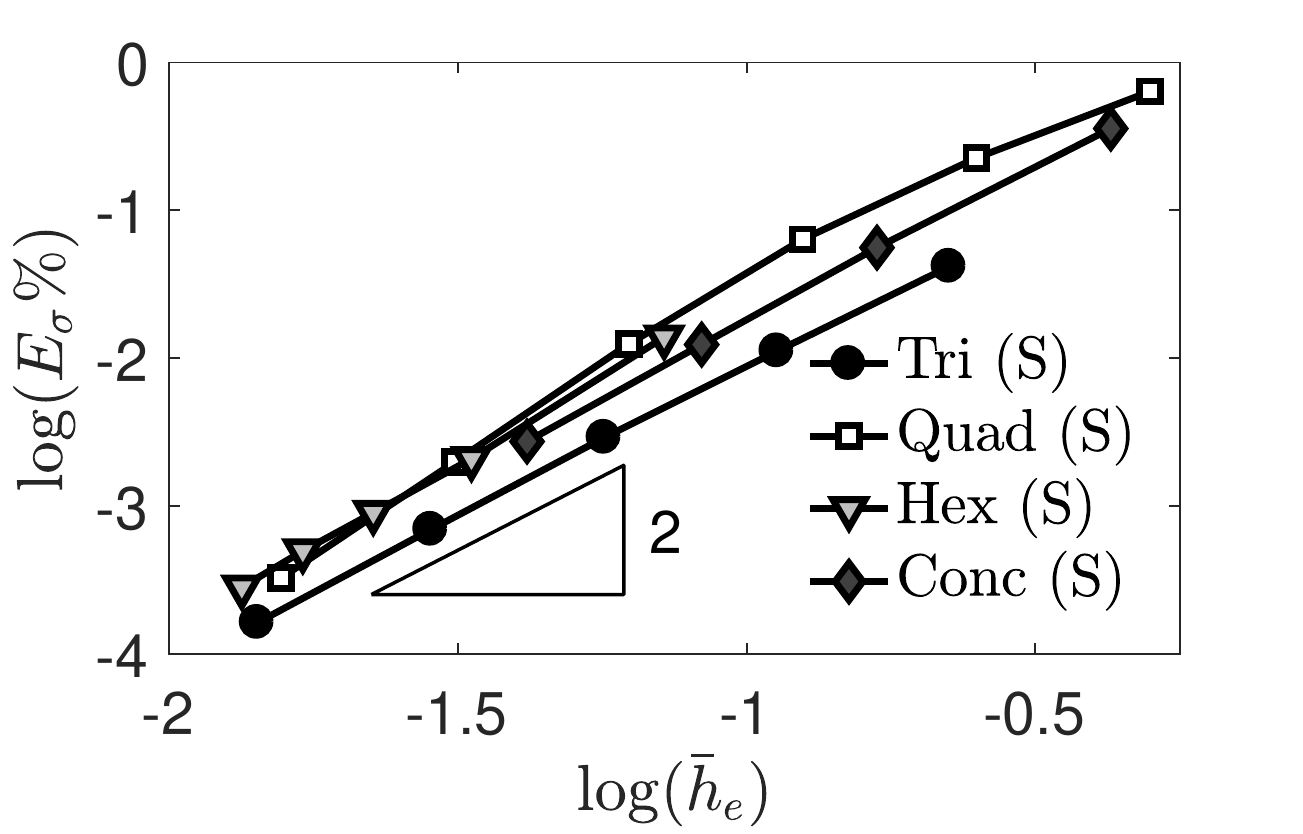}}
	\subfigure[]{\includegraphics[width=0.45\textwidth,trim = 0mm 0mm 0mm 0mm, clip]{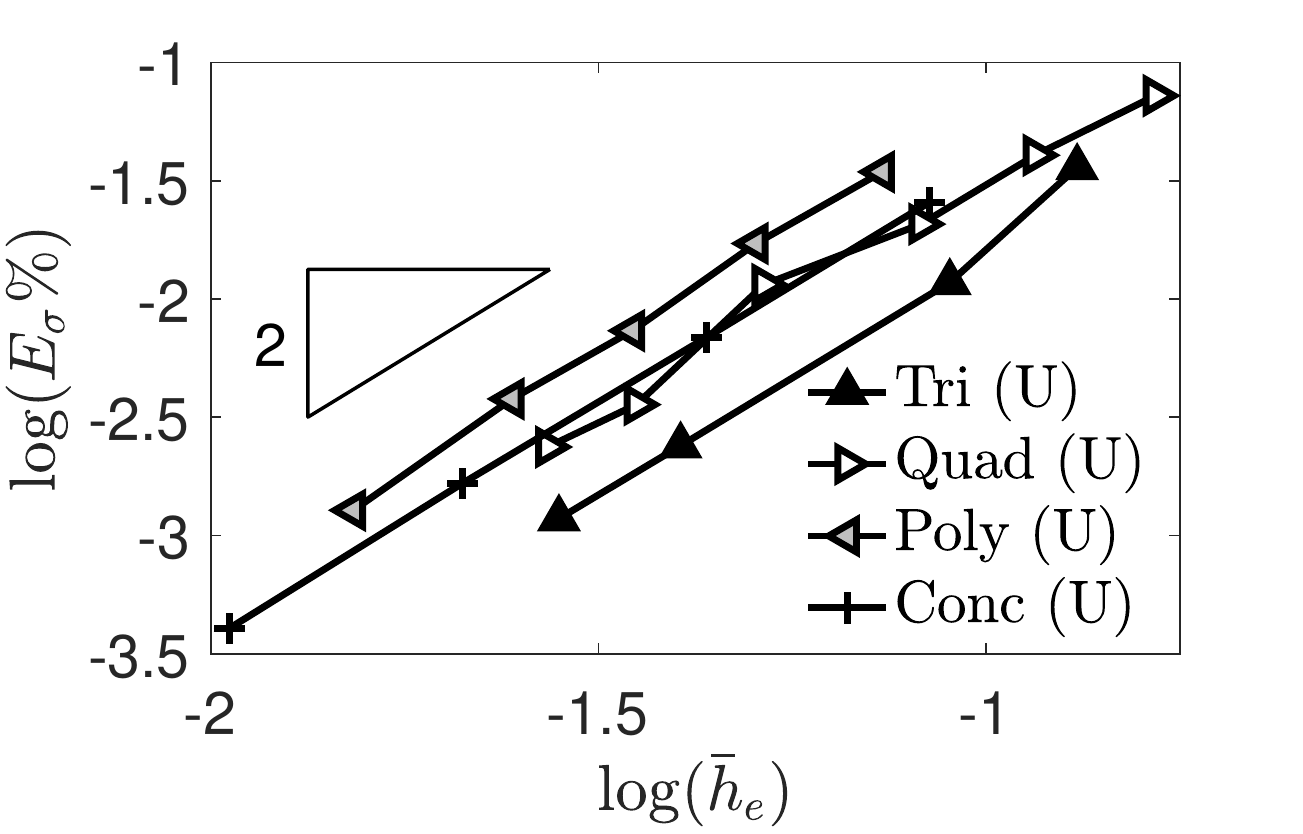}}\\
	\subfigure[]{\includegraphics[width=0.45\textwidth,trim = 0mm 0mm 0mm 0mm, clip]{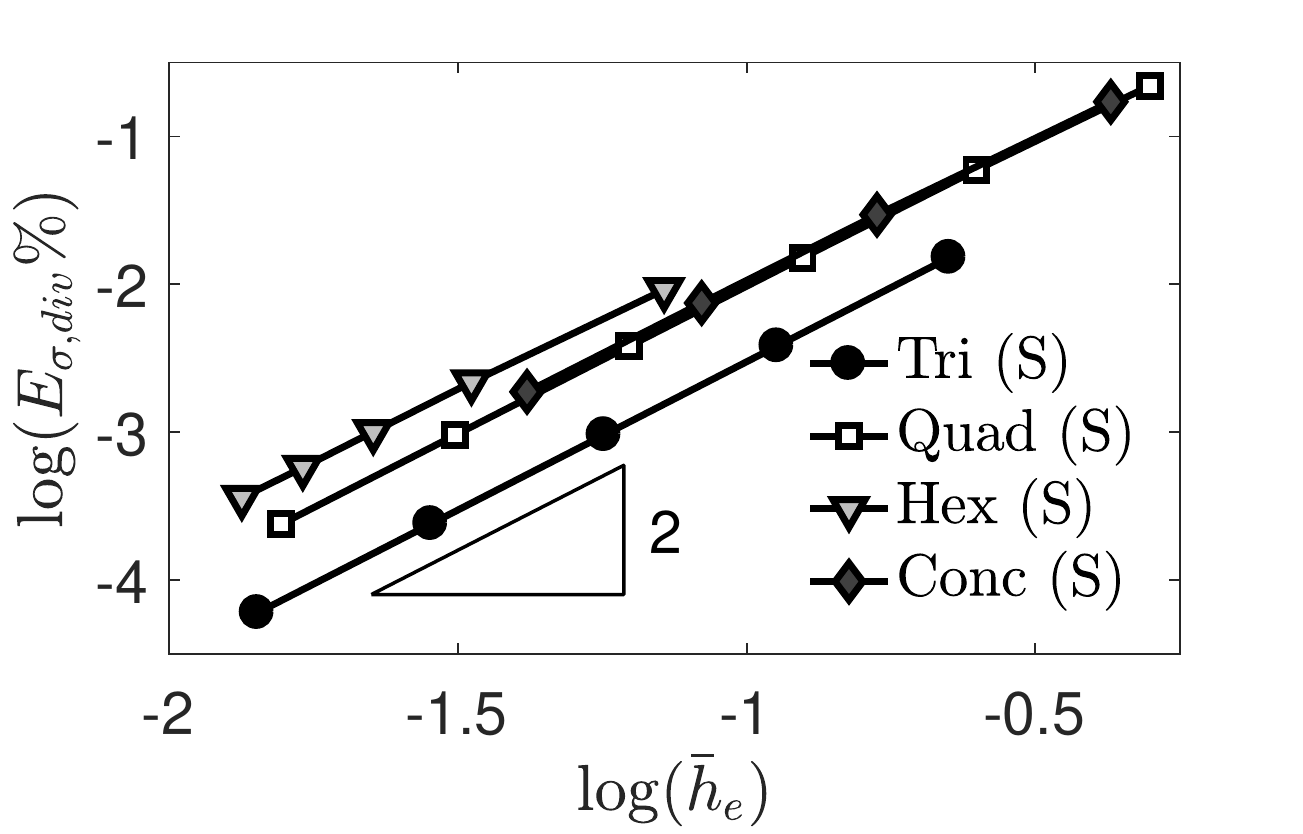}}
	\subfigure[]{\includegraphics[width=0.45\textwidth,trim = 0mm 0mm 0mm 0mm, clip]{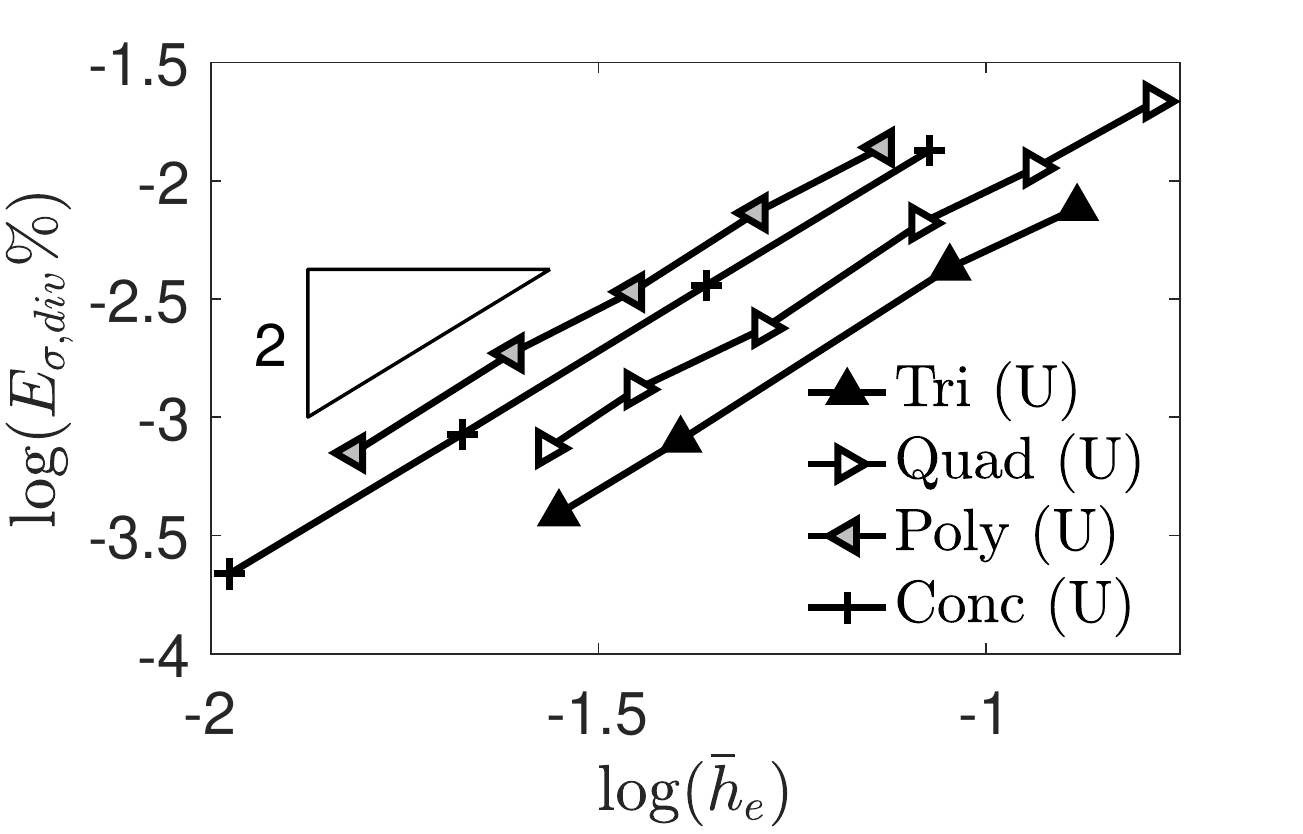}}\\
	\subfigure[]{\includegraphics[width=0.45\textwidth,trim = 0mm 0mm 0mm 0mm, clip]{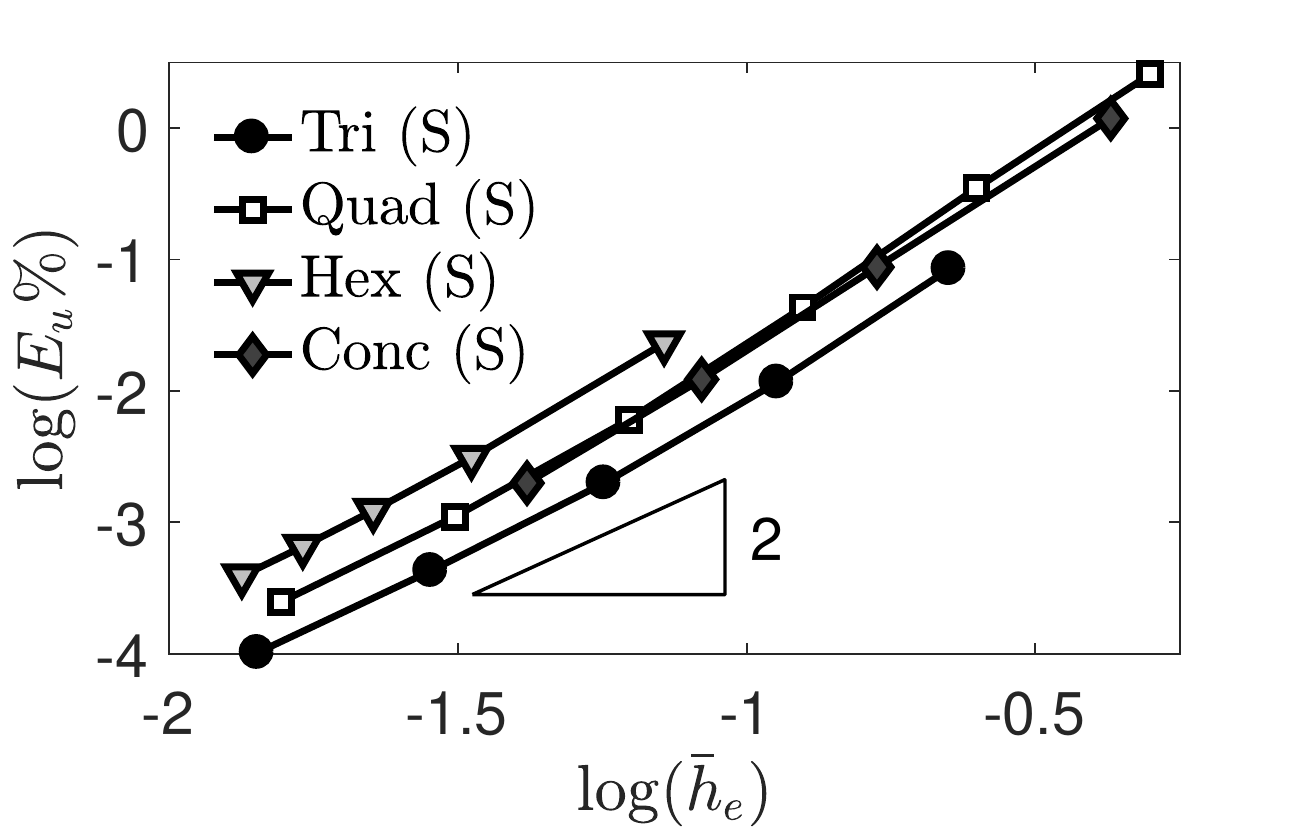}}
	\subfigure[]{\includegraphics[width=0.45\textwidth,trim = 0mm 0mm 0mm 0mm, clip]{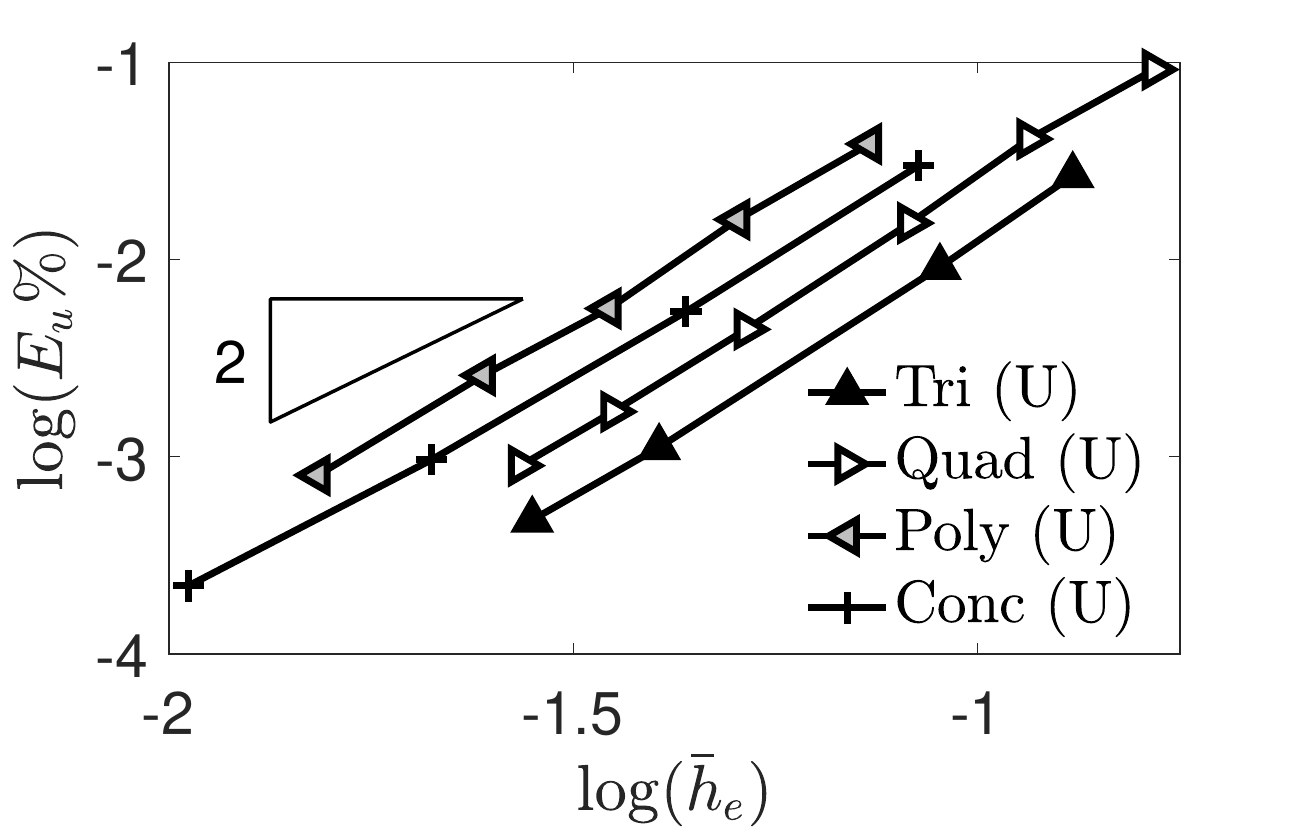}}\\
	\caption{$\bar{h}_e-$convergence results for Test $b$ on structured and unstructured meshes for $k=1$: (a) and (b) $E_{\bfsigma}$ error norm plots, (c) and (d) $E_{\bfsigma, \bdiv}$ error norm plots, (e) and (f) $E_{\bbu}$ error norm plots.}
	\label{fig:resuTestB}
\end{figure}

\begin{figure}[h!]
	\centering
	\subfigure[]{\includegraphics[width=0.45\textwidth,trim = 0mm 0mm 0mm 0mm, clip]{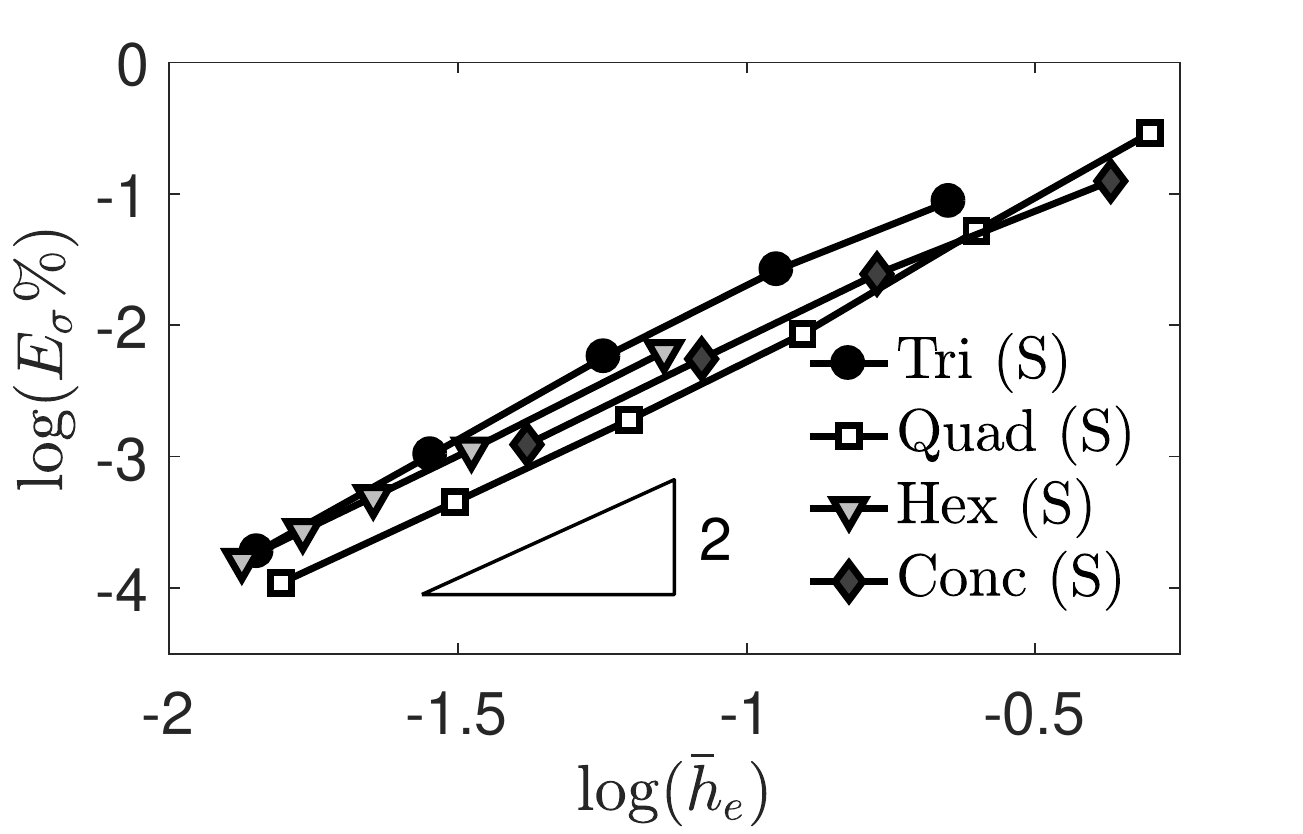}}
	\subfigure[]{\includegraphics[width=0.45\textwidth,trim = 0mm 0mm 0mm 0mm, clip]{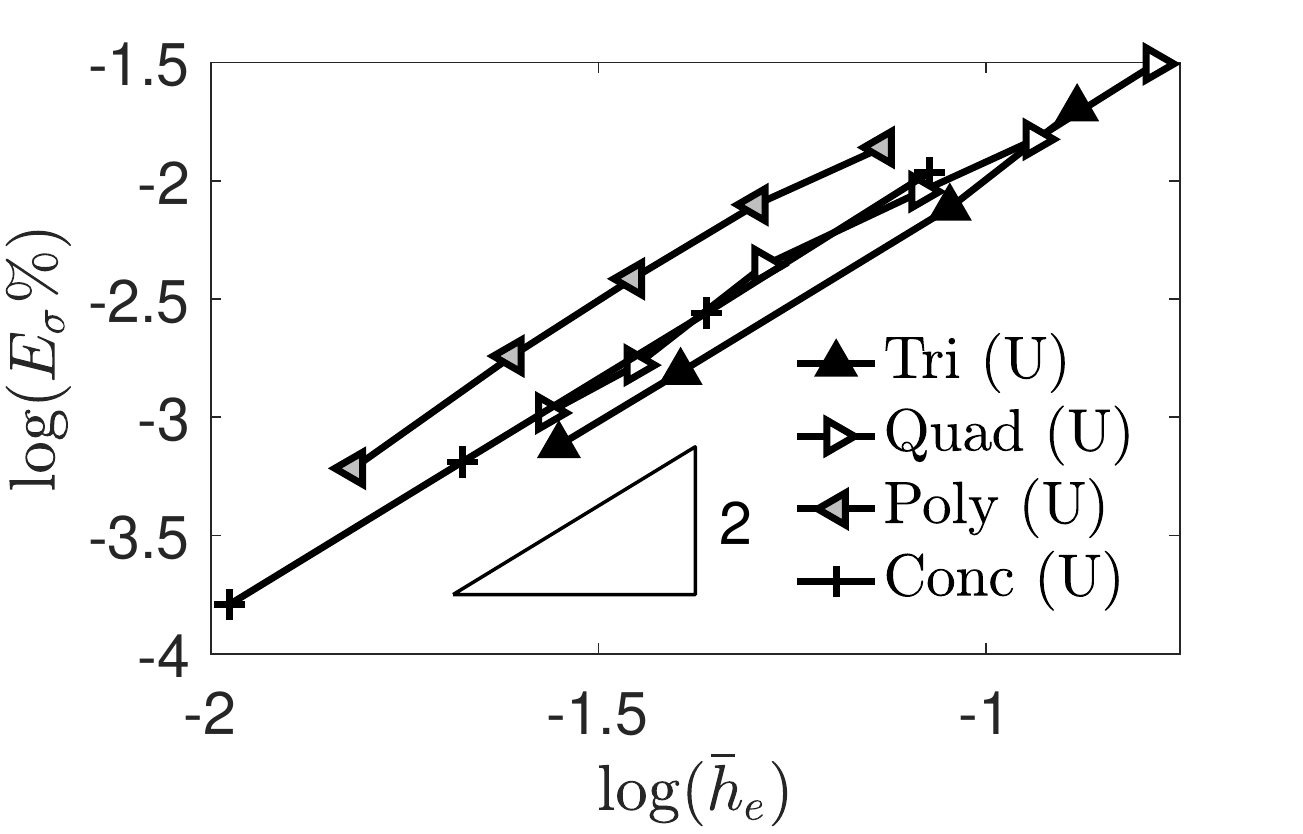}}\\
	\subfigure[]{\includegraphics[width=0.45\textwidth,trim = 0mm 0mm 0mm 0mm, clip]{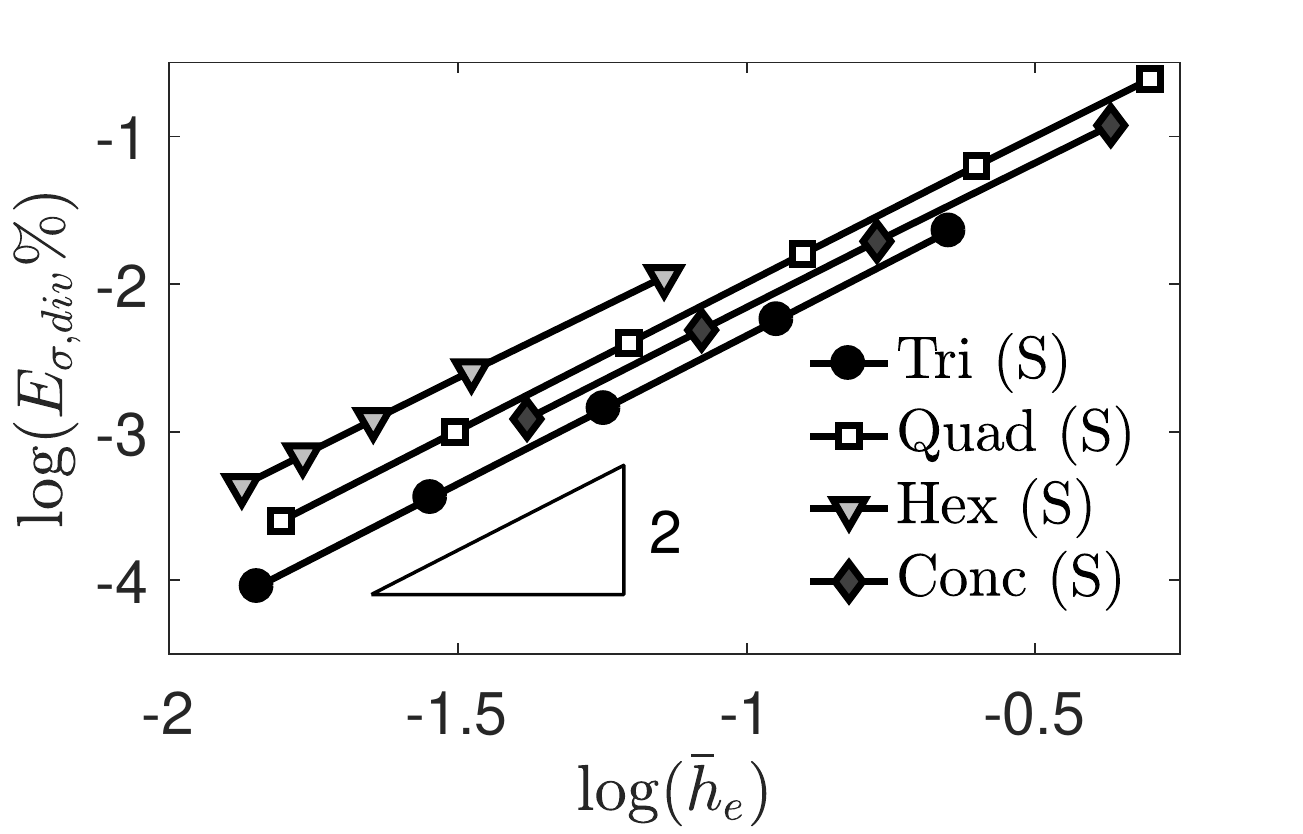}}
	\subfigure[]{\includegraphics[width=0.45\textwidth,trim = 0mm 0mm 0mm 0mm, clip]{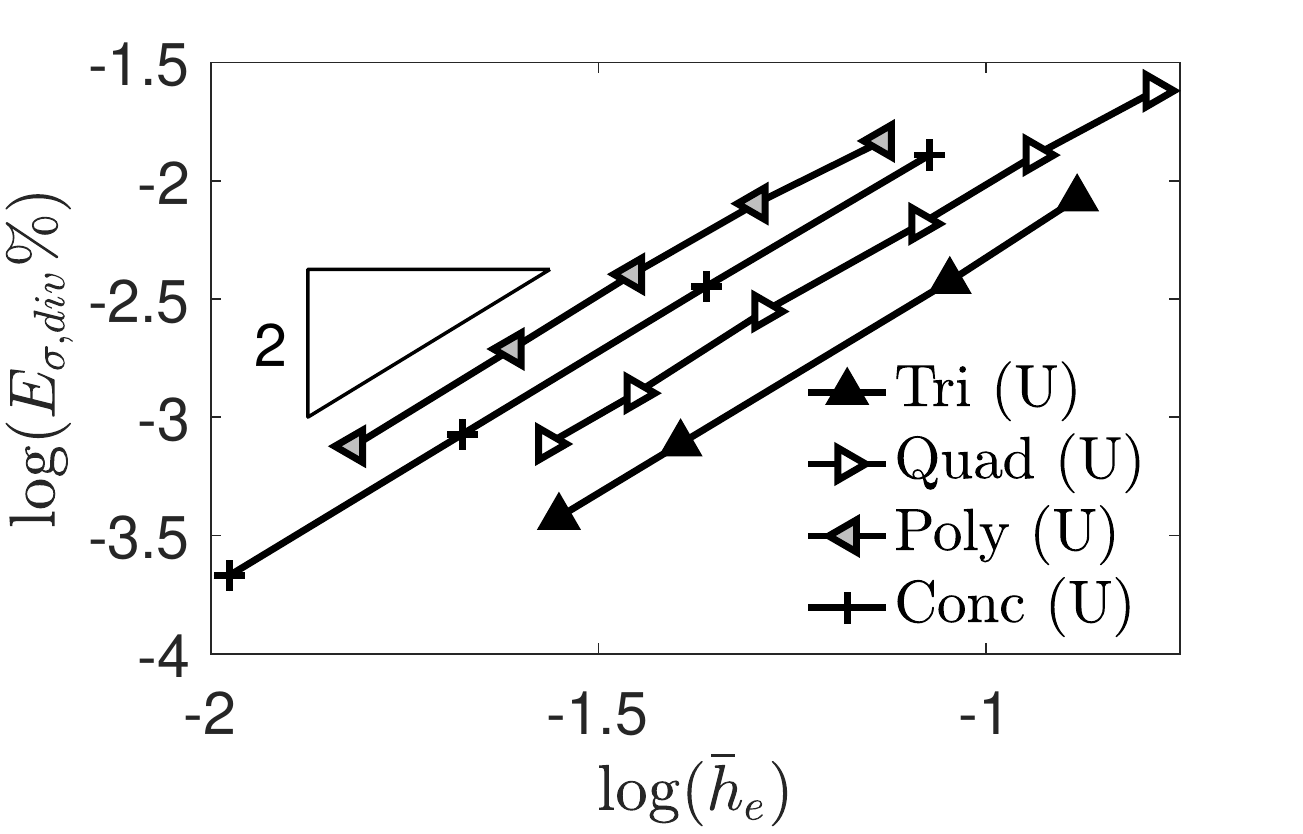}}\\
	\subfigure[]{\includegraphics[width=0.45\textwidth,trim = 0mm 0mm 0mm 0mm, clip]{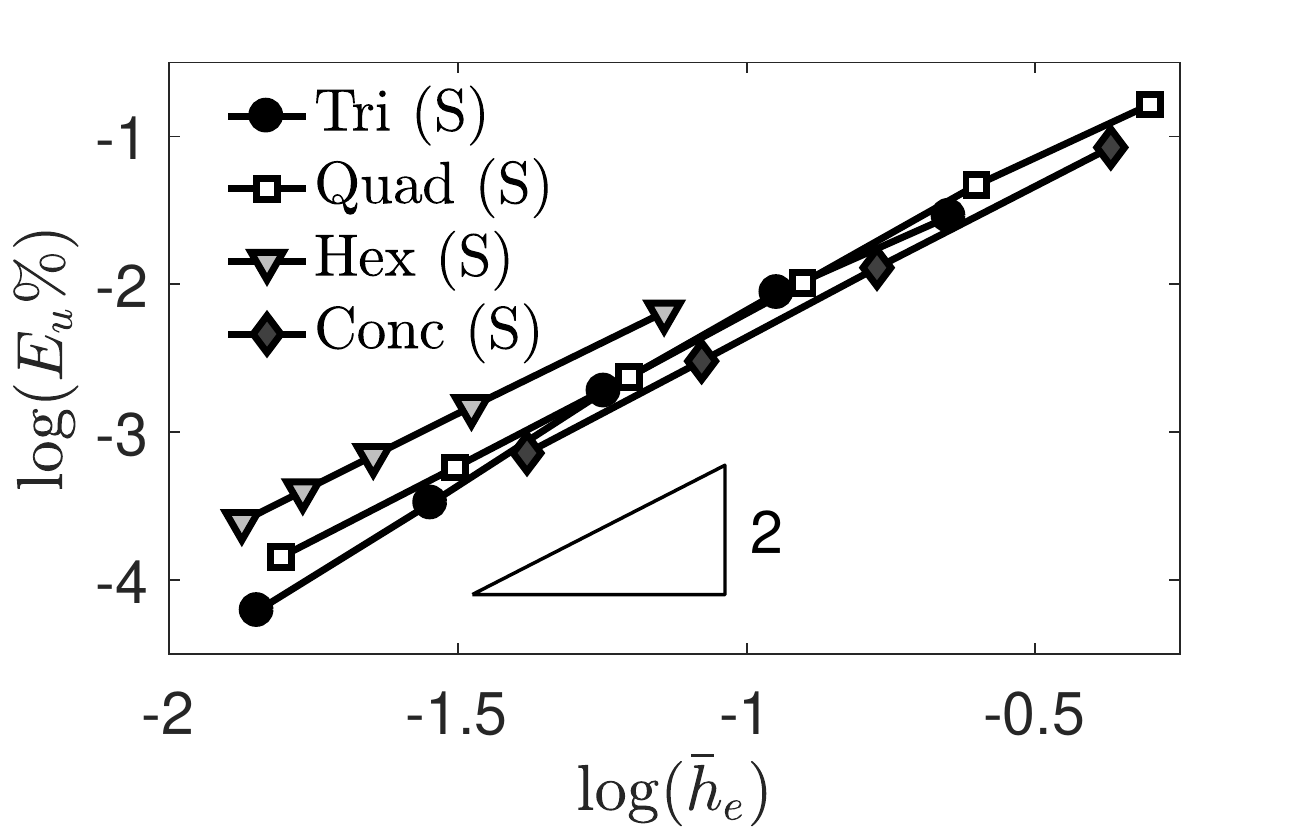}}
	\subfigure[]{\includegraphics[width=0.45\textwidth,trim = 0mm 0mm 0mm 0mm, clip]{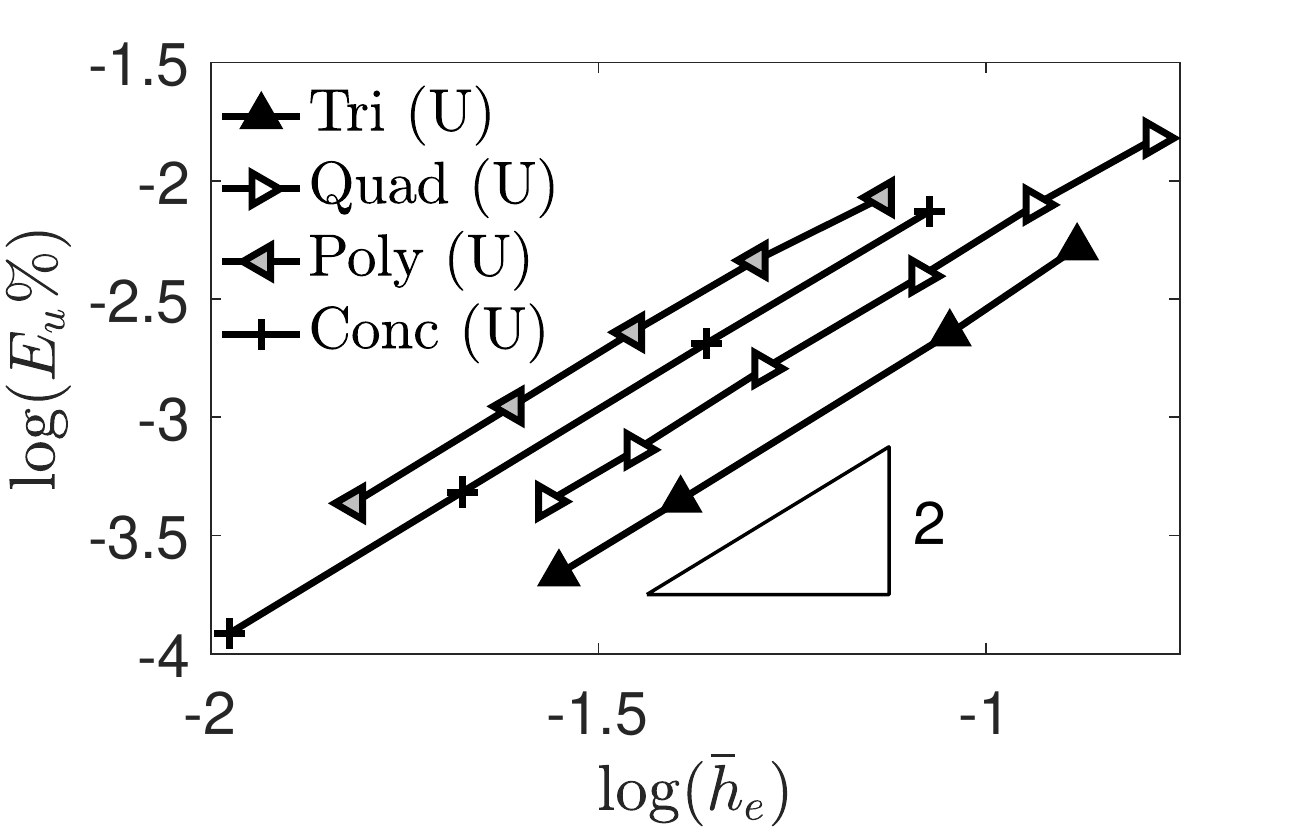}}\\
	\caption{$\bar{h}_e-$convergence results for Test $c$ on structured and unstructured meshes for $k=1$: (a) and (b) $E_{\bfsigma}$ error norm plots, (c) and (d) $E_{\bfsigma, \bdiv}$ error norm plots, (e) and (f) $E_{\bbu}$ error norm plots.}
	\label{fig:resuTestAvar}
\end{figure}

\subsection{Results for $k=2$}
Figure \ref{fig:resuTestA2} reports the $\bar{h}_e-$convergence of the proposed method for Test $a$ when $k$ is equal to 2. The asymptotic convergence rate is approximately equal to $3$ for all the considered error norms and meshes, as expected. In this case the $E_{\bfsigma}$ and $E_{\bfsigma, \bdiv}$ plots are not reported because such quantities are captured up to machine precision for all the considered computational grids.

\begin{figure}[h!]
	\centering
	\subfigure[]{\includegraphics[width=0.45\textwidth,trim = 0mm 0mm 0mm 0mm, clip]{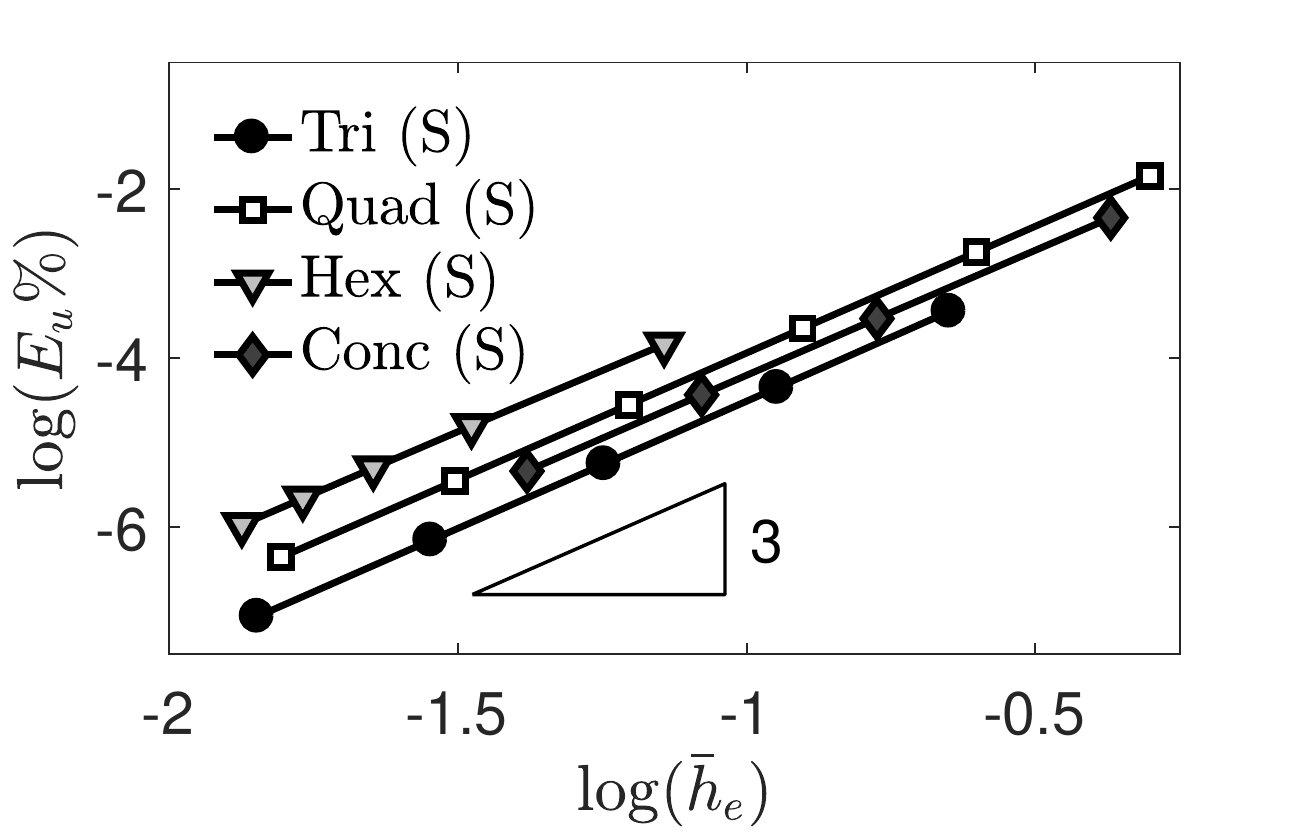}}
	\subfigure[]{\includegraphics[width=0.45\textwidth,trim = 0mm 0mm 0mm 0mm, clip]{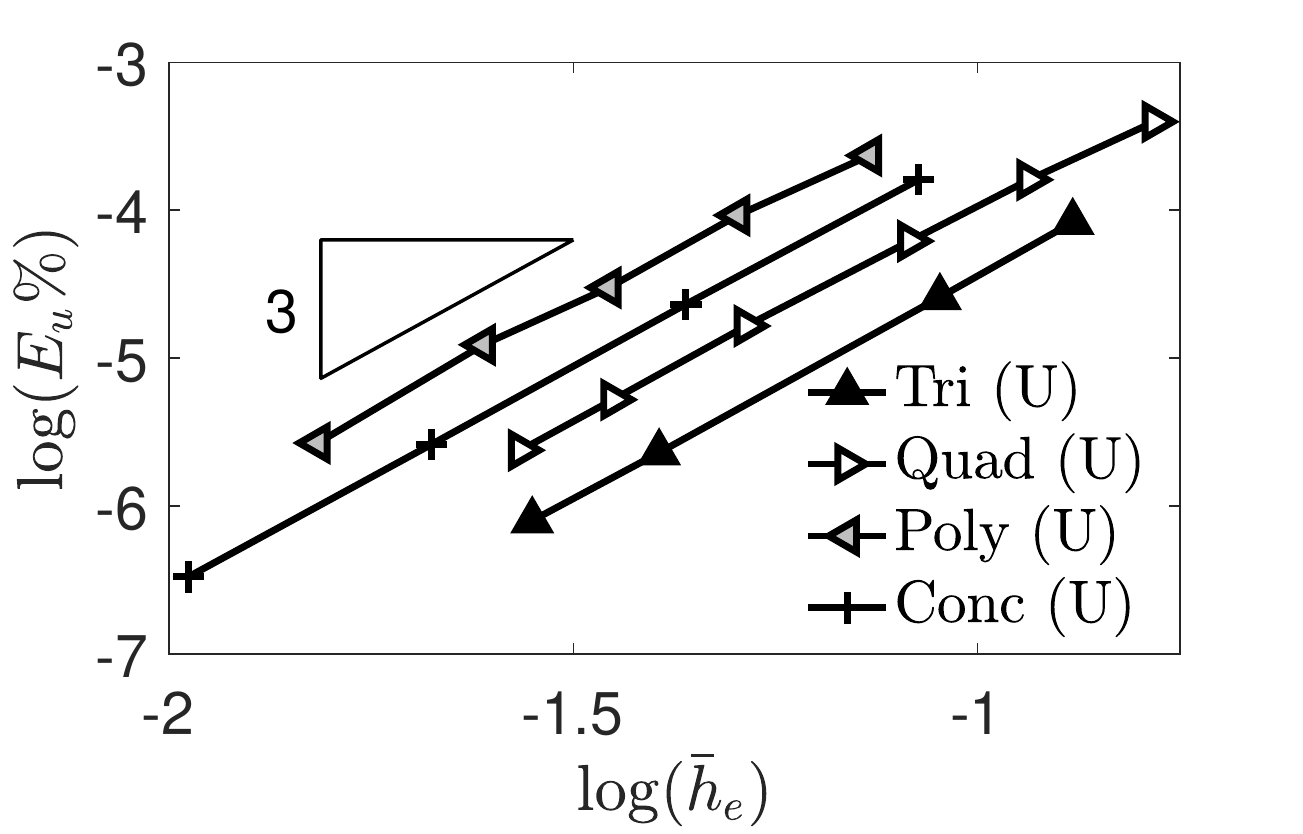}}
	\caption{$\bar{h}_e-$convergence results for Test $a$ on structured and unstructured meshes for $k=2$: (a) and (b) $E_{\bbu}$ error norm plots.}
	\label{fig:resuTestA2}
\end{figure}

Figure \ref{fig:resuTestB2} and \ref{fig:resuTestAvar2} report $\bar{h}_e-$convergence for Test $b$ and Test $c$. Also in this case results confirm the soundness of the proposed approach.

\begin{figure}[h!]
	\centering
	\subfigure[]{\includegraphics[width=0.45\textwidth,trim = 0mm 0mm 0mm 0mm, clip]{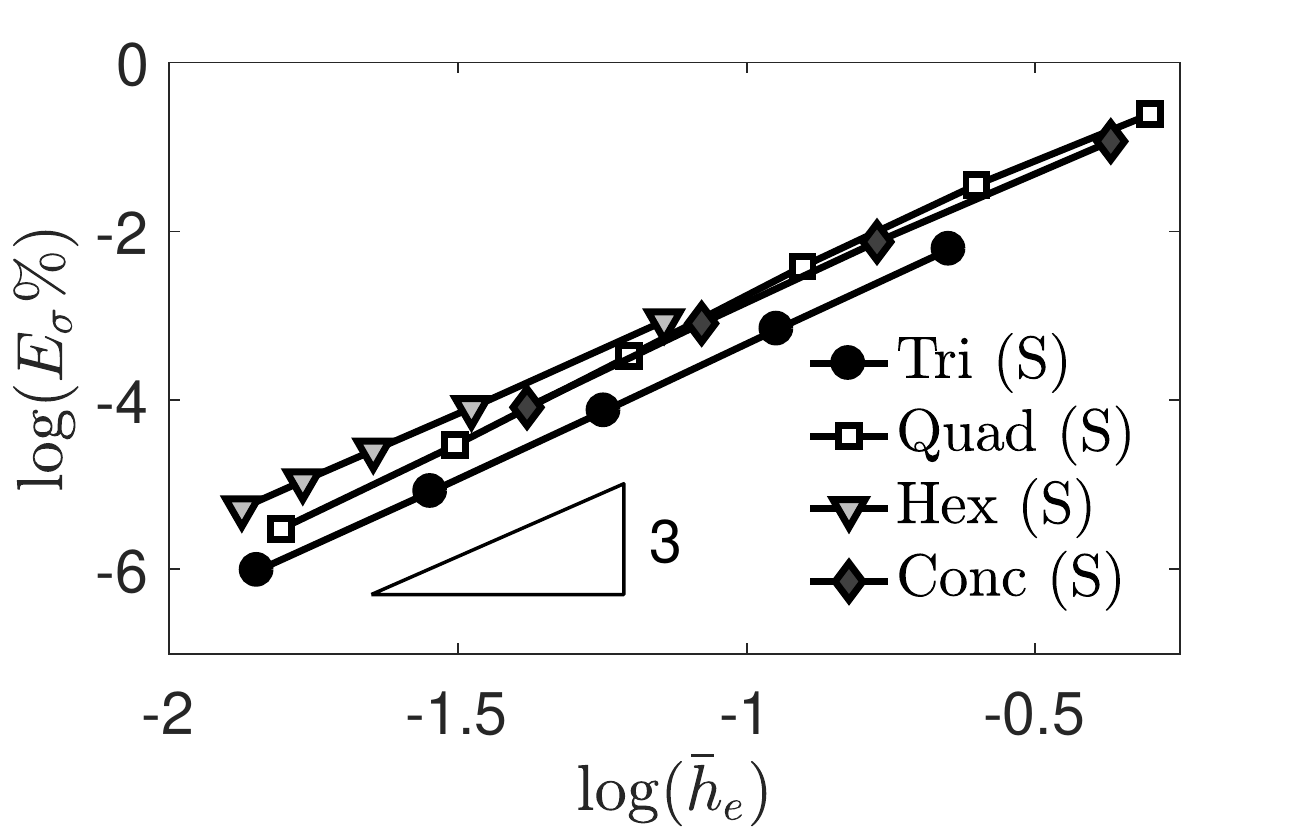}}
	\subfigure[]{\includegraphics[width=0.45\textwidth,trim = 0mm 0mm 0mm 0mm, clip]{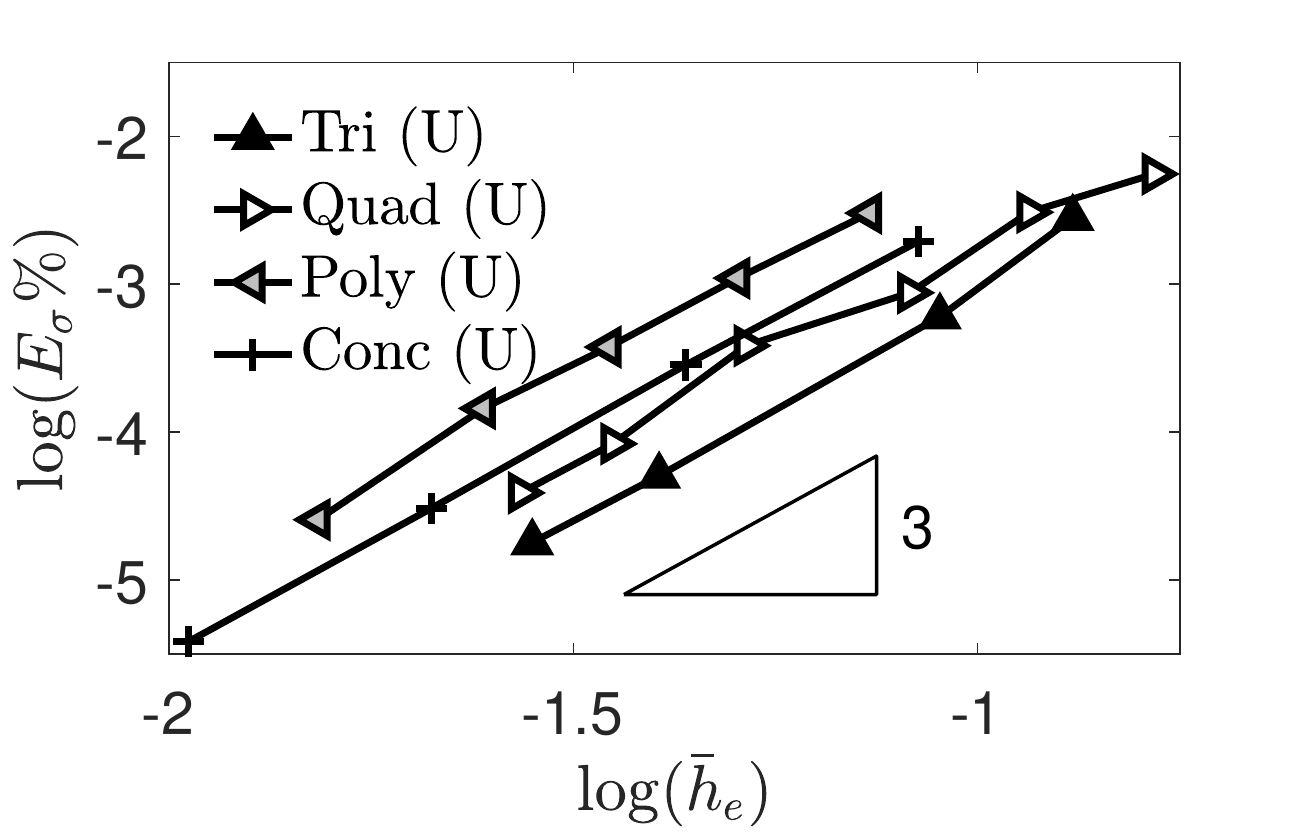}}\\
	\subfigure[]{\includegraphics[width=0.45\textwidth,trim = 0mm 0mm 0mm 0mm, clip]{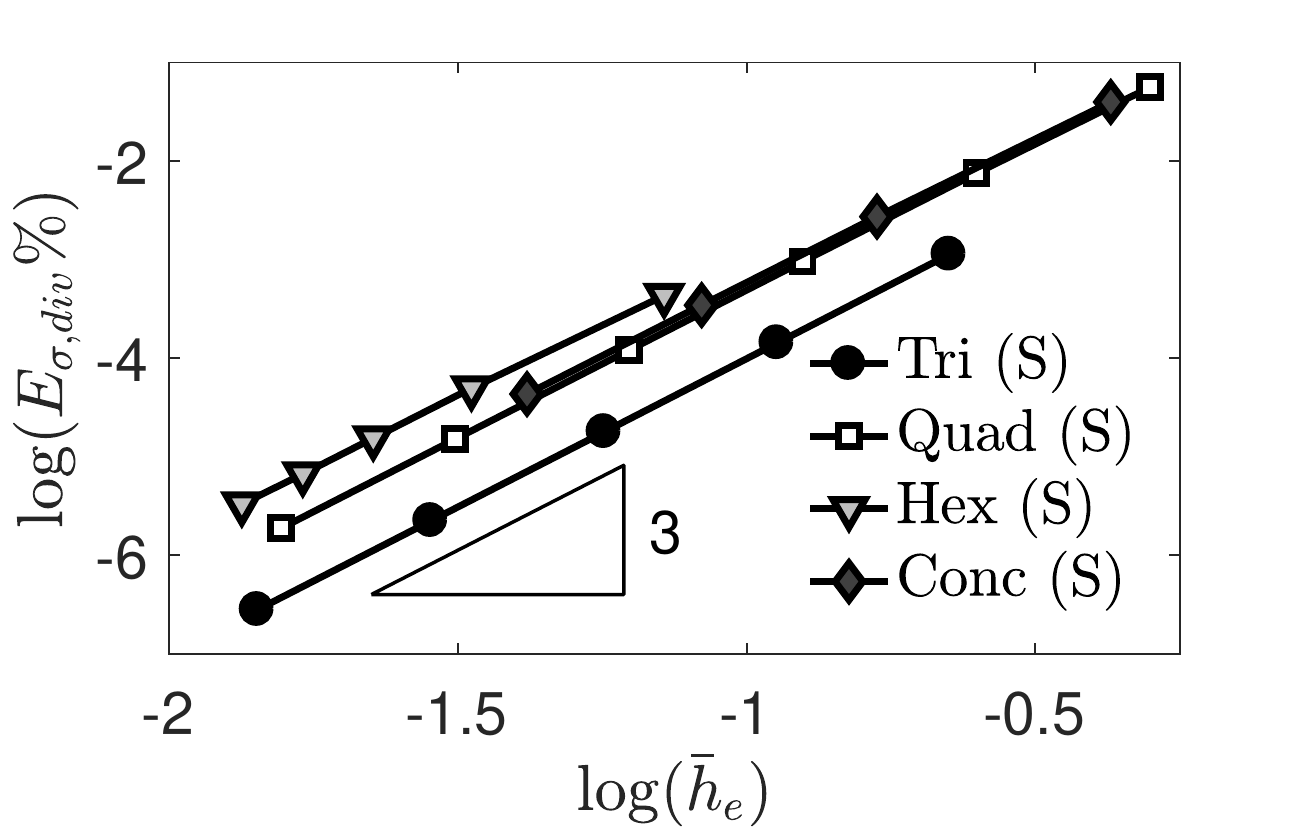}}
	\subfigure[]{\includegraphics[width=0.45\textwidth,trim = 0mm 0mm 0mm 0mm, clip]{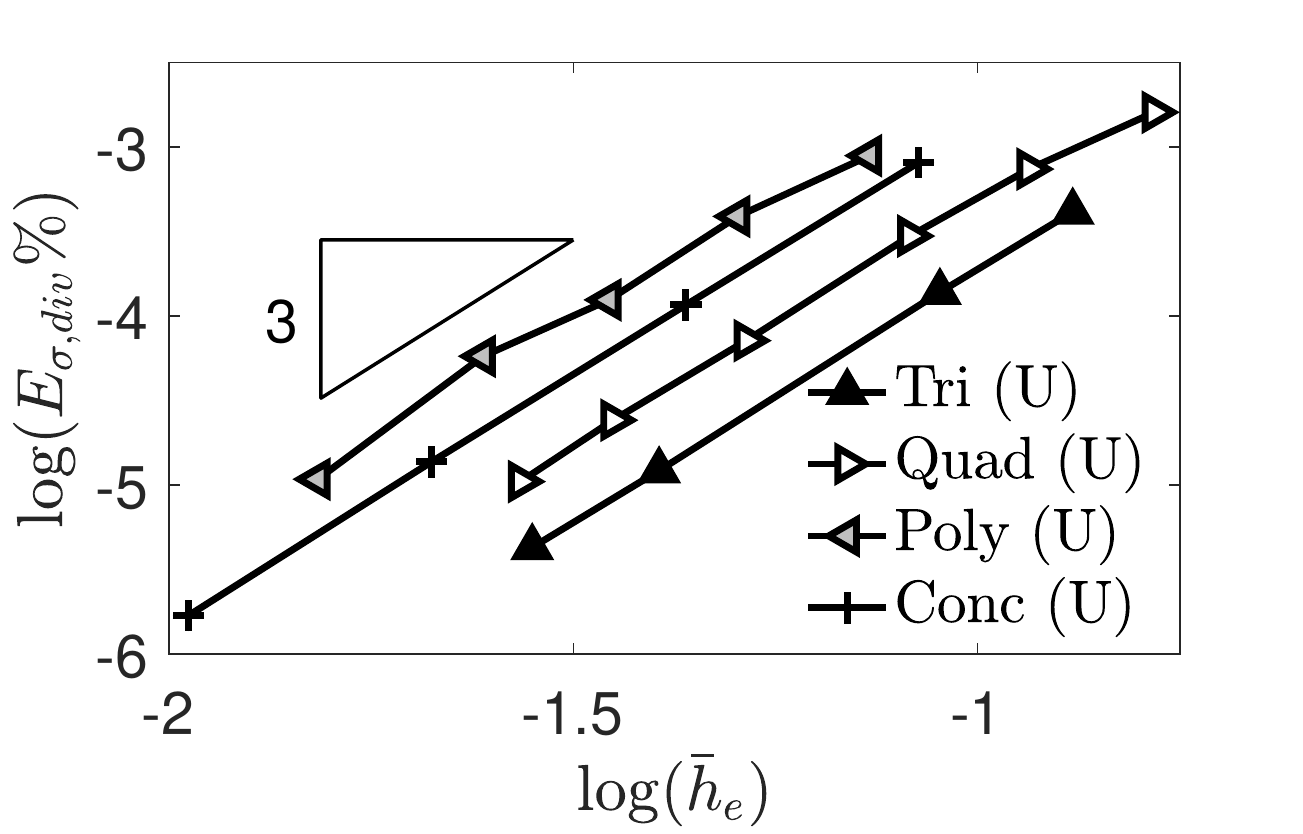}}\\
	\subfigure[]{\includegraphics[width=0.45\textwidth,trim = 0mm 0mm 0mm 0mm, clip]{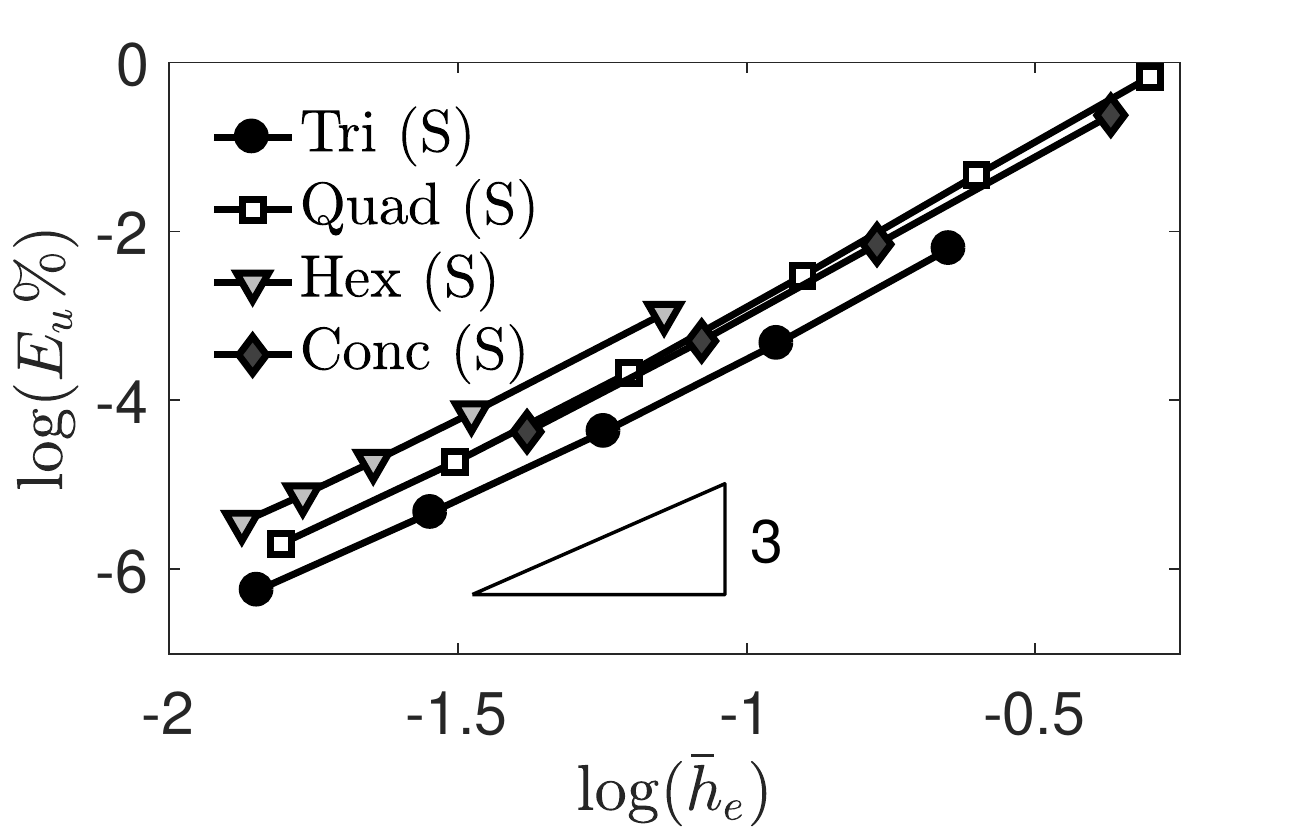}}
	\subfigure[]{\includegraphics[width=0.45\textwidth,trim = 0mm 0mm 0mm 0mm, clip]{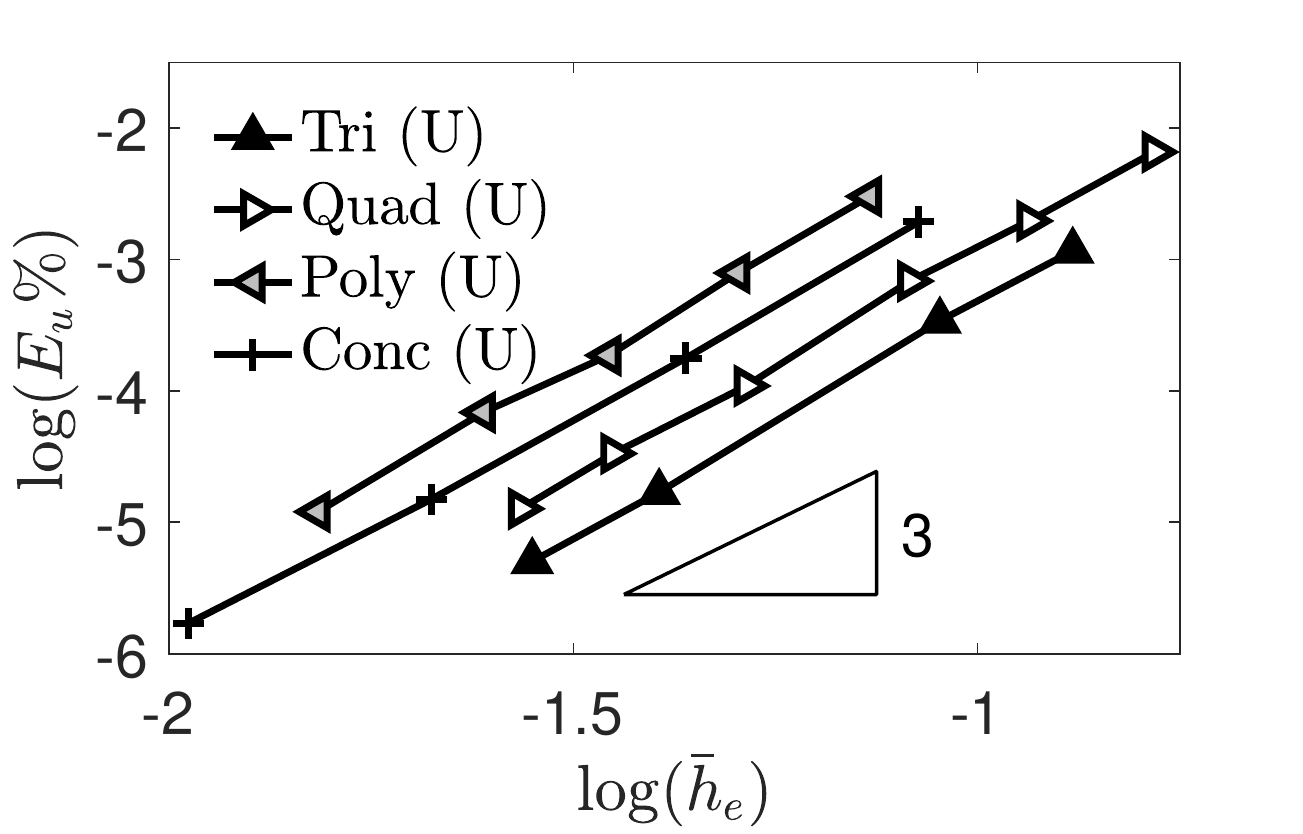}}\\
	\caption{$\bar{h}_e-$convergence results for Test $b$ on structured and unstructured meshes for $k=2$: (a) and (b) $E_{\bfsigma}$ error norm plots, (c) and (d) $E_{\bfsigma, \bdiv}$ error norm plots, (e) and (f) $E_{\bbu}$ error norm plots.}
	\label{fig:resuTestB2}
\end{figure}

\begin{figure}[h!]
	\centering
	\subfigure[]{\includegraphics[width=0.45\textwidth,trim = 0mm 0mm 0mm 0mm, clip]{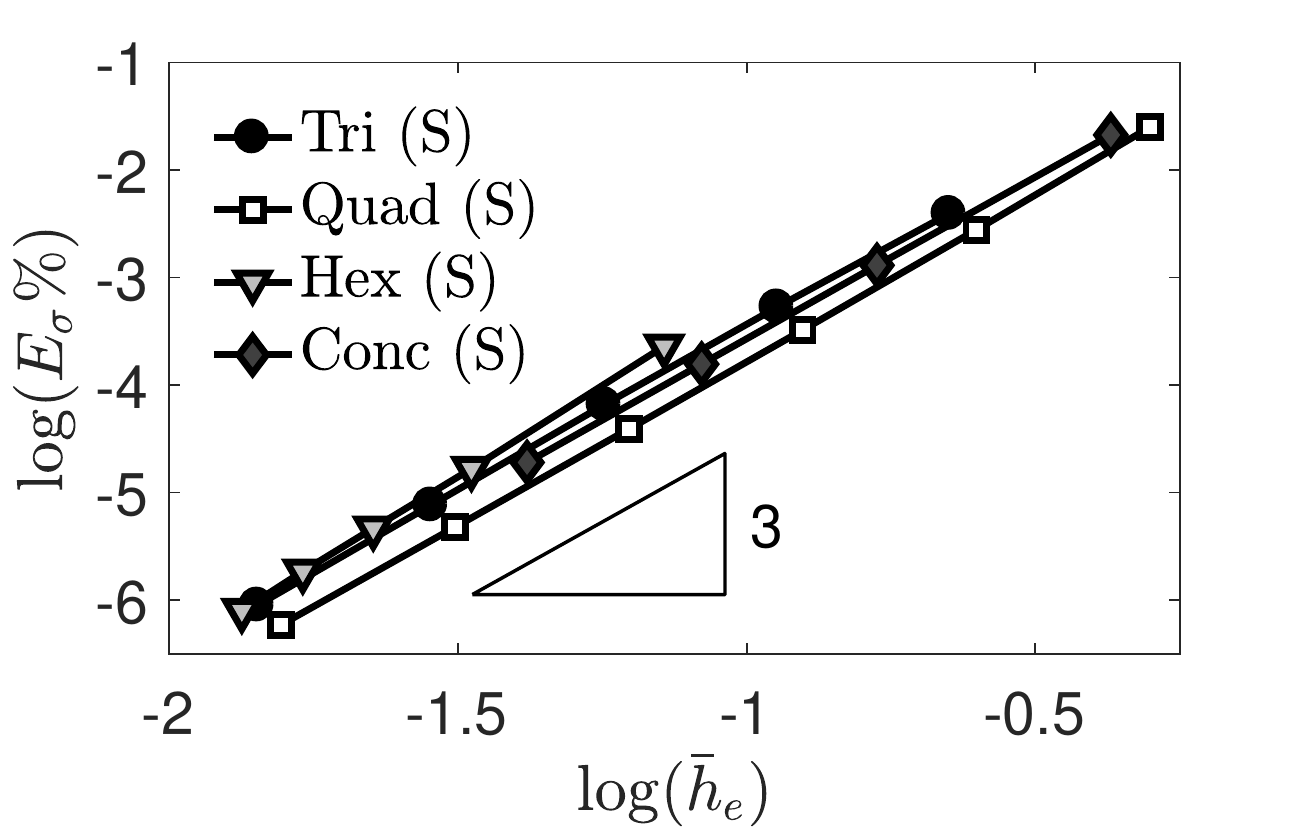}}
	\subfigure[]{\includegraphics[width=0.45\textwidth,trim = 0mm 0mm 0mm 0mm, clip]{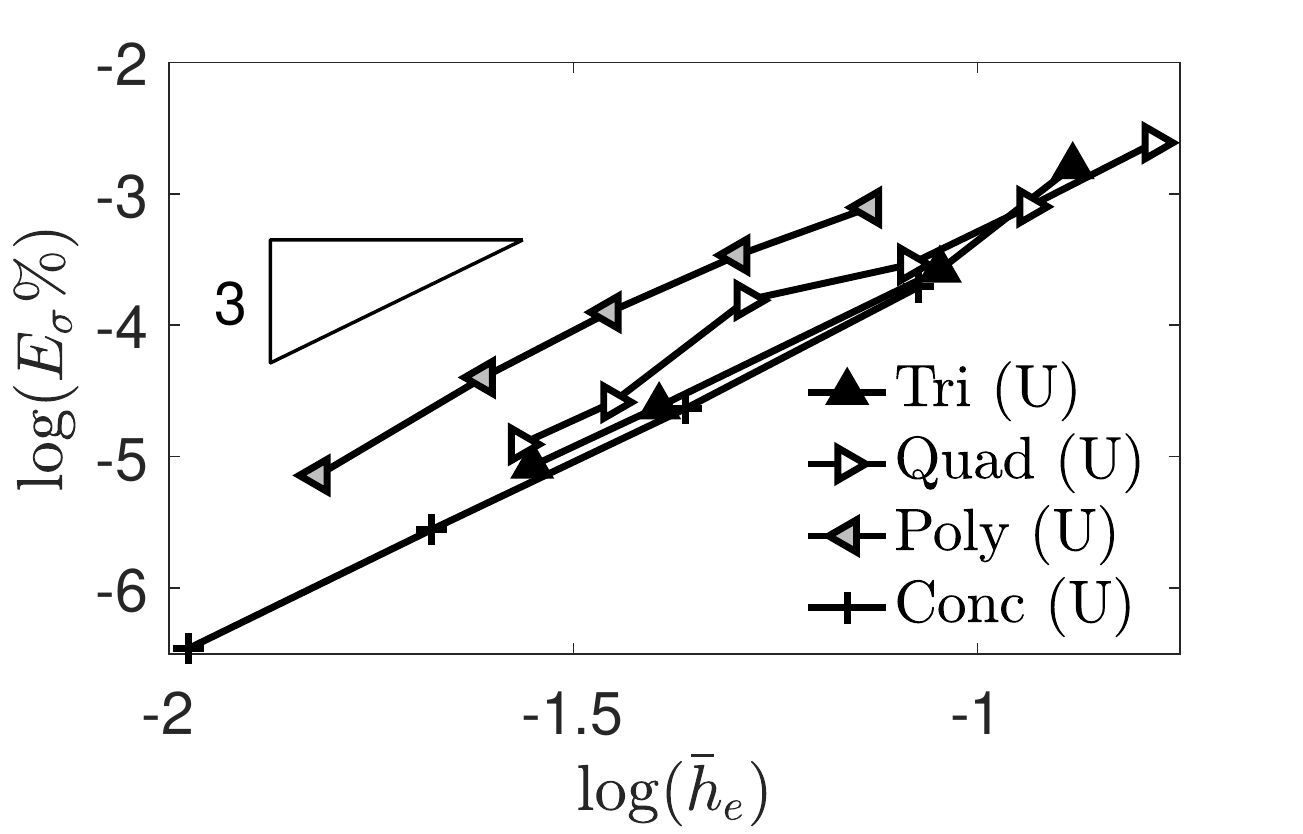}}\\
	\subfigure[]{\includegraphics[width=0.45\textwidth,trim = 0mm 0mm 0mm 0mm, clip]{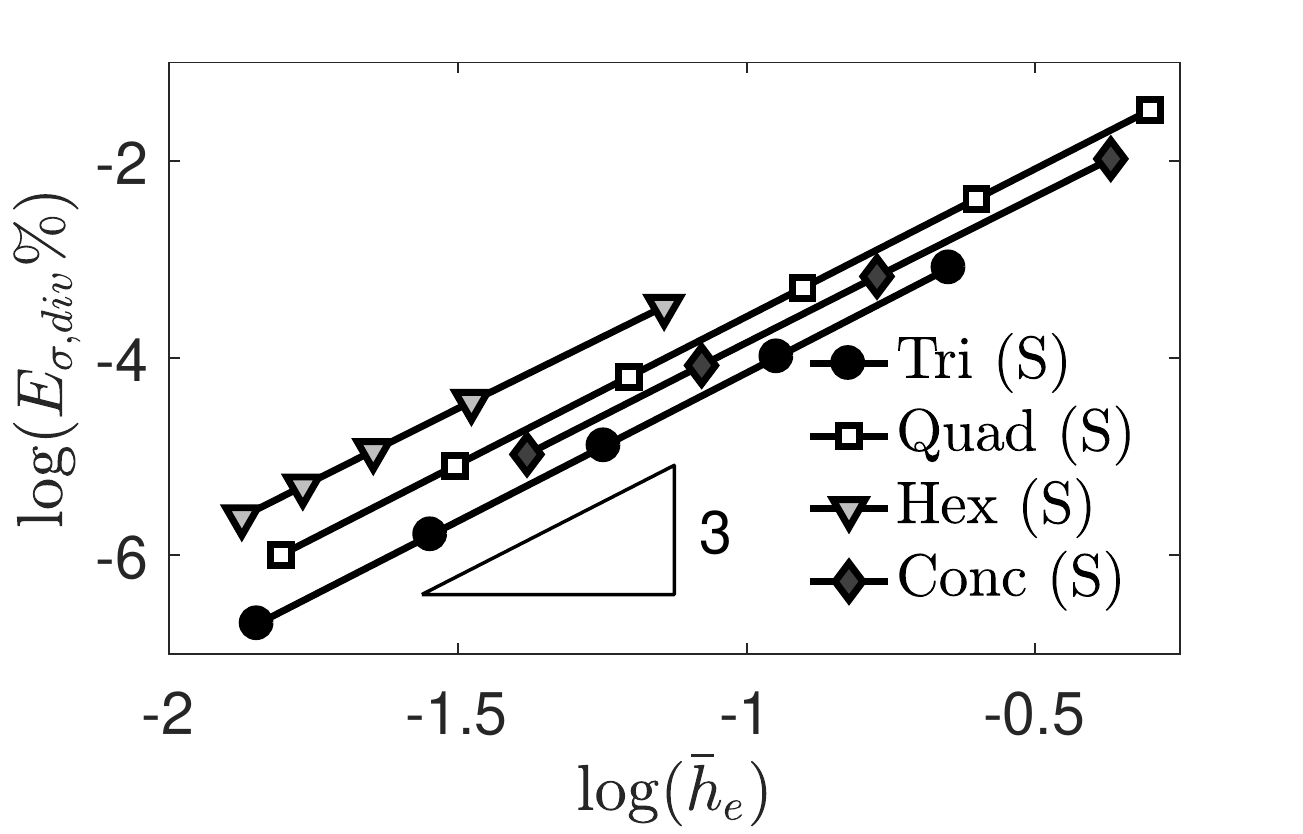}}
	\subfigure[]{\includegraphics[width=0.45\textwidth,trim = 0mm 0mm 0mm 0mm, clip]{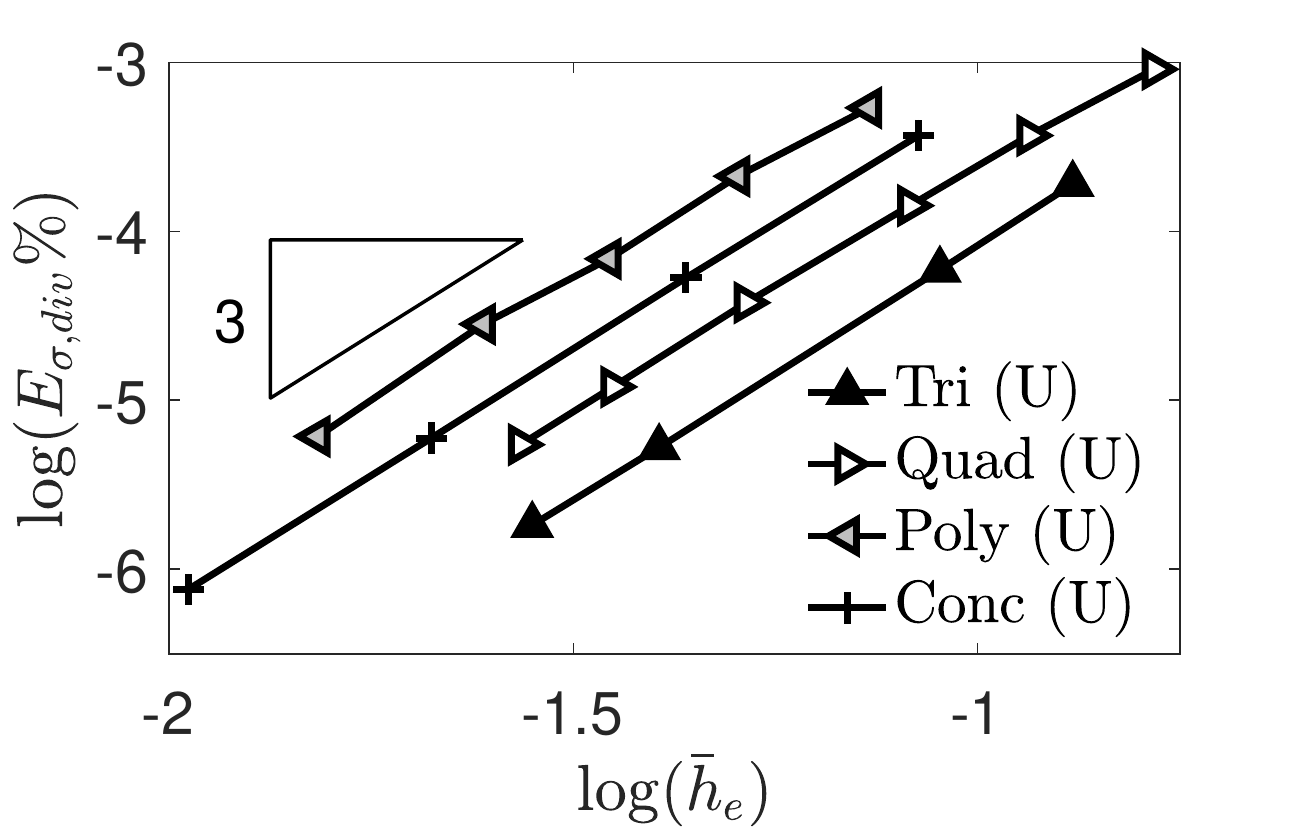}}\\
	\subfigure[]{\includegraphics[width=0.45\textwidth,trim = 0mm 0mm 0mm 0mm, clip]{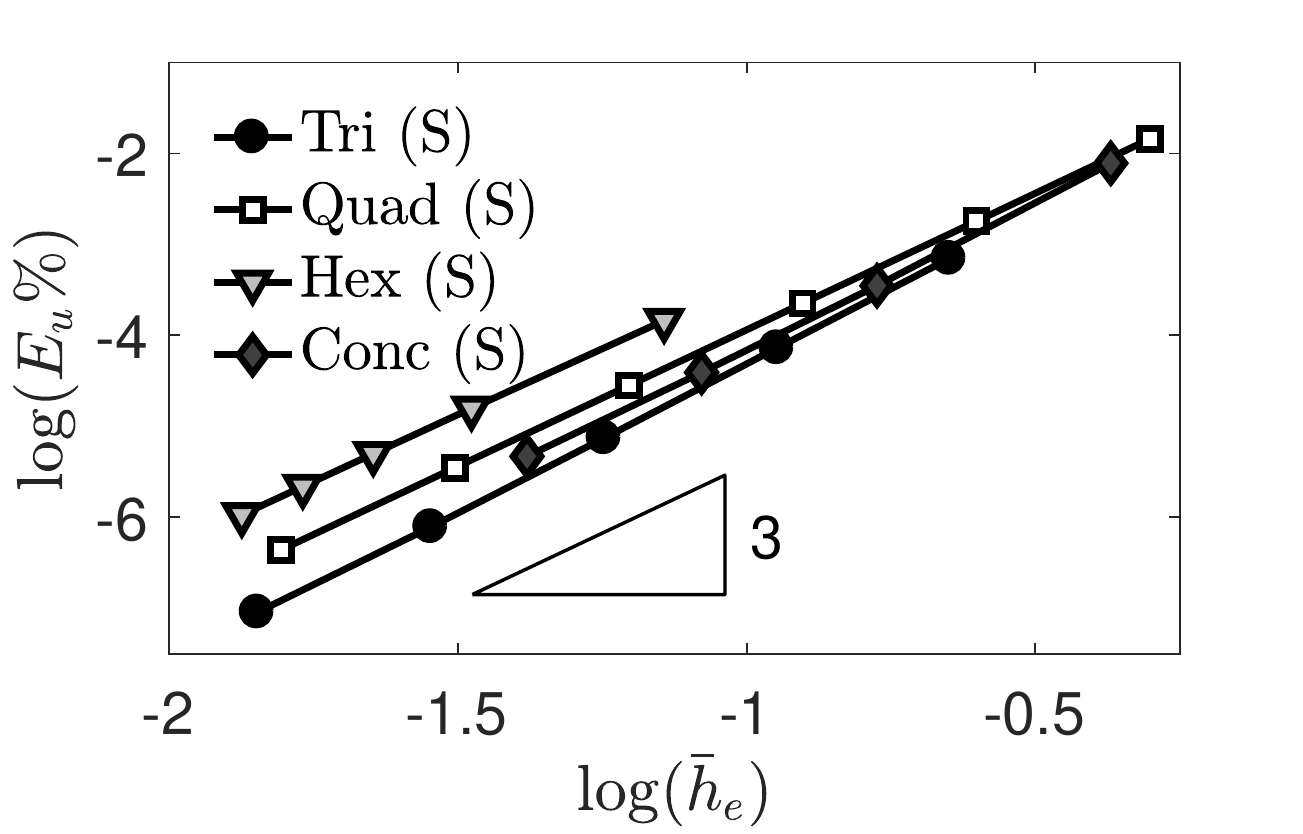}}
	\subfigure[]{\includegraphics[width=0.45\textwidth,trim = 0mm 0mm 0mm 0mm, clip]{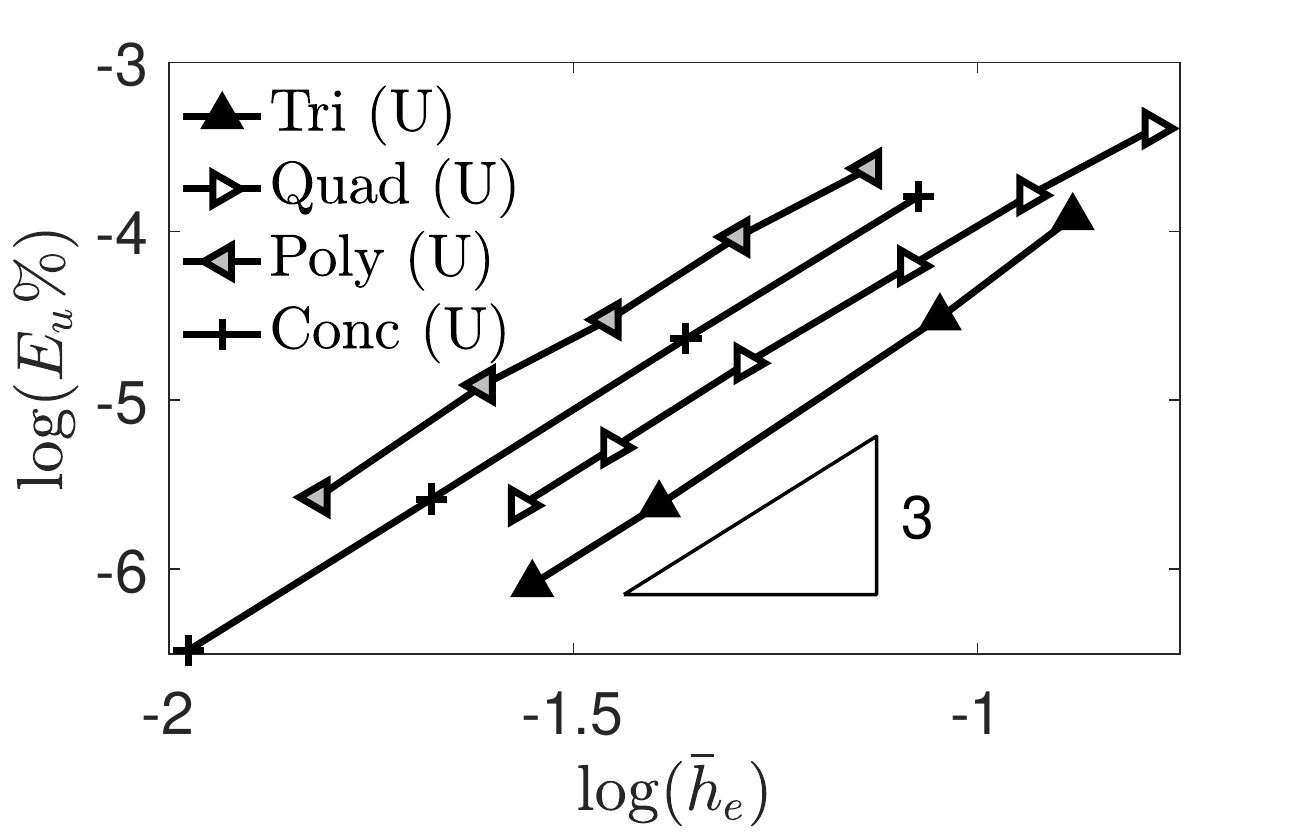}}\\
	\caption{$\bar{h}_e-$convergence results for Test $c$ on structured and unstructured meshes for $k=2$: (a) and (b) $E_{\bfsigma}$ error norm plots, (c) and (d) $E_{\bfsigma, \bdiv}$ error norm plots, (e) and (f) $E_{\bbu}$ error norm plots.}
	\label{fig:resuTestAvar2}
\end{figure}


\section{Conclusions}\label{s:conclusions}

We have presented a family of Virtual Element schemes for the linear elasticity 2D problem, described by the mixed Hellinger-Reissner variational principle. The approximated stresses are {\em a priori} symmetric and the corresponding tractions are continuous across the polygon inter-elements. 
We have proved that our methods are stable and optimal convergent, and we have reported some numerical tests that confirm the theoretical predictions.
A possible interesting evolution of this paper could be the extension of the present approach to the {\em three-dimensional} case.

\medskip

%


\bibliographystyle{amsplain}

\bibliography{general-bibliography,biblio,VEM}


\end{document}